\documentclass[10pt,a4paper]{article}
\usepackage{amsmath,amsfonts,amsthm}
\usepackage{fullpage}
\usepackage{authblk}
\usepackage{bbold}
\usepackage{booktabs}
\usepackage{comment}
\usepackage{cite}
\usepackage{enumerate}
\usepackage{general}
\usepackage{hyperref}
\usepackage[bold]{hhtensor}
\usepackage{multirow}
\usepackage{nicefrac}
\usepackage{paralist}
\usepackage{slopetriangle}
\usepackage{subcaption}
\usepackage{xcolor}

\graphicspath{{figures/pdf/}}

\newcommand{\email}[1]{\href{mailto:#1}{#1}}

\usepackage{parskip}

\makeatletter
\def\thm@space@setup{%
  \thm@preskip=\parskip \thm@postskip=0pt
}
\makeatother

\usetikzlibrary{external}
\newcommand{\U}{\vec{U}}
\newcommand{\Uh}[1][k]{\underline{\vec{U}}_h^{#1}}
\newcommand{\UhD}[1][k]{\underline{\vec{U}}_{h,0}^{#1}}
\newcommand{\UT}[1][k]{\underline{\vec{U}}_T^{#1}}
\newcommand{\Ph}[1][k]{P_h^{#1}}
\newcommand{\Wh}[1][k]{W_h^{#1}}

\newcommand{\Ih}[1][k]{\underline{\vec{I}}_h^{#1}}
\newcommand{\IT}[1][k]{\underline{\vec{I}}_T^{#1}}

\newcommand{\lproj}[2][h]{\pi_{#1}^{#2}}
\newcommand{\vproj}[2][h]{\vec{\pi}_{#1}^{#2}}
\newcommand{\mproj}[2][h]{\matr{\pi}_{#1}^{#2}}

\newcommand{\uu}[1][T]{\underline{\vec{u}}_{#1}}
\newcommand{\uhu}[1][T]{\widehat{\underline{\vec{u}}}_{#1}}
\newcommand{\uv}[1][T]{\underline{\vec{v}}_{#1}}
\newcommand{\uhv}[1][T]{\widehat{\underline{\vec{v}}}_{#1}}
\newcommand{\uw}[1][T]{\underline{\vec{w}}_{#1}}

\newcommand{\vu}{\vec{u}}
\newcommand{\vv}{\vec{v}}
\newcommand{\vw}{\vec{w}}
\newcommand{\vf}{\vec{f}}

\newcommand{\uxi}[1][T]{\underline{\vec{\xi}}_{#1}}

\newcommand{\vPhi}{\vec{\Phi}}
\newcommand{\vhPhi}{\widehat{\vec{\Phi}}}
\newcommand{\uvhPhi}[1][h]{\widehat{\underline{\vec{\Phi}}}_{#1}}

\newcommand{\mtau}{\matr{\tau}}

\newcommand{\vhu}{\widehat{\vec{u}}}
\newcommand{\vhv}{\widehat{\vec{v}}}

\newcommand{\hph}{\widehat{p}_h}

\newcommand{\ue}[1][h]{\underline{\vec{e}}_{#1}}
\newcommand{\eps}[1][h]{\epsilon_{#1}}

\newcommand{\Eh}{\mathcal{E}_h}

\newcommand{\GT}[1][k]{\matr{G}_T^{#1}}
\newcommand{\GsT}[1][k]{\matr{G}_{{\rm s},T}^{#1}}
\newcommand{\Gh}[1][k]{\matr{G}_h^{#1}}
\newcommand{\Gsh}[1][k]{\matr{G}_{{\rm s},h}^{#1}}
\newcommand{\pT}[1][k+1]{\vec{r}_T^{#1}}
\newcommand{\ph}[1][k+1]{\vec{r}_h^{#1}}
\newcommand{\dTF}[1][k]{\vec{\delta}_{TF}^{#1}}
\newcommand{\DT}[1][k]{D_T^{#1}}
\newcommand{\Dh}[1][k]{D_h^{#1}}

\newcommand{\MVPROD}{\,}

\newcommand{\Reynolds}{\rm Re}

\title{A Hybrid High-Order method for the steady incompressible Navier--Stokes problem\footnote{The work of D. A. Di Pietro was partially supported by Agence Nationale de la Recherche project HHOMM (ANR-15-CE40-0005).}}
\author[1]{Daniele A. Di Pietro\footnote{\email{daniele.di-pietro@umontpellier.fr}}}
\affil[1]{Universit\'{e} de Montpellier, Institut Montpelli\'{e}rain Alexander Grothendieck, 34095 Montpellier, France}
\author[2]{Stella Krell\footnote{stella.krell@unice.fr}}
\affil[2]{Universit\'{e} de Nice, Laboratoire J. A. Dieudonn\'{e}, avenue Valrose, 06000 Nice, France}

\begin{document}

\maketitle

\begin{abstract}
  In this work we introduce and analyze a novel Hybrid High-Order method for the steady incompressible Navier--Stokes equations.
  The proposed method is inf-sup stable on general polyhedral meshes, supports arbitrary approximation orders, and is (relatively) inexpensive thanks to the possibility of statically condensing a subset of the unknowns at each nonlinear iteration.
  We show under general assumptions the existence of a discrete solution, which is also unique provided a data smallness condition is verified.
  Using a compactness argument, we prove convergence of the sequence of discrete solutions to minimal regularity exact solutions for general data.
  For more regular solutions, we prove optimal convergence rates for the energy-norm of the velocity and the $L^2$-norm of the pressure under a standard data smallness assumption.
  More precisely, when polynomials of degree $k\ge 0$ at mesh elements and faces are used, both quantities are proved to converge as $h^{k+1}$ (with $h$ denoting the meshsize).
  %% Additionally, we numerically show convergence of the $L^2$-norm of the velocity as $h^{k+2}$.
\end{abstract}

\section{Introduction}

In this work we introduce and analyze a novel Hybrid High-Order (HHO) method for the steady incompressible Navier--Stokes equations.
The proposed method is inf-sup stable on general meshes including polyhedral elements and nonmatching interfaces, it supports arbitrary approximation order, and has a reduced computational cost thanks to the possibility of statically condensing a subset of both velocity and pressure degrees of freedom (DOFs) at each nonlinear iteration.
A complete analysis covering general data is provided.

Let $\Omega\subset\Real^d$, $d\in\{2,3\}$, denote a bounded connected open polyhedral domain.
The incompressible Navier--Stokes problem consists in finding the velocity field $\vu:\Omega\to\Real^d$ and the pressure field $p:\Omega\to\Real$ such that
\begin{subequations}\label{eq:strong}
  \begin{alignat}{2}
    \label{eq:strong:momentum}
    -\nu\LAPL\vu + \GRAD\vu\MVPROD\vu + \GRAD p &= \vf &\qquad& \text{in $\Omega$},
    \\
    \label{eq:strong:mass}
    \DIV\vu &= 0 &\qquad& \text{in $\Omega$},
    \\
    \label{eq:strong:bc}
    \vu &= \vec{0} &\qquad& \text{on $\partial\Omega$},
    \\
    \label{eq:strong:zero-mean}
    \int_\Omega p &= 0,
  \end{alignat}
\end{subequations}
where $\nu>0$ denotes the (constant) kinematic viscosity and $\vf\in L^2(\Omega)^d$ a volumetric body force.
For the convective term, here and in what follows we will use the matrix-vector product notation with vector quantities treated as column vectors, so that, e.g., $\GRAD\vu\MVPROD\vu = \sum_{j=1}^d(\partial_j u_i)u_j$ with $u_i$ denoting the $i$th component of $\vu$ and $\partial_j$ the derivative along the $j$th direction.
Let $\U\eqbydef H_0^1(\Omega)^d$ and $P\eqbydef L^2_0(\Omega)$ with $L^2_0(\Omega)\eqbydef\left\{q\in L^2(\Omega)\st\int_\Omega q = 0\right\}$.
A classical weak formulation of problem \eqref{eq:strong} reads:
Find $(\vu,p)\in\U\times P$ such that
\begin{subequations}\label{eq:weak}
  \begin{alignat}{2}\label{eq:weak:momentum}
    \nu a(\vu,\vv) + t(\vu,\vu,\vv) + b(\vv,p) &= \int_\Omega\vf\SCAL\vv &\qquad& \forall\vv\in\U,
    \\ \label{eq:weak:mass}
    -b(\vu,q) &= 0 &\qquad& \forall q\in P,
  \end{alignat}
\end{subequations}
with bilinear forms $a:\U\times\U\to\Real$ and $b:\U\times P\to\Real$ and trilinear form $t:\U\times\U\times\U\to\Real$ such that
\begin{equation}\label{eq:a.b.t}
  a(\vu,\vv)\eqbydef\int_\Omega\GRAD\vu\SSCAL\GRAD\vv,\qquad
  b(\vv,q)\eqbydef-\int_\Omega(\DIV\vv) q,\qquad
  t(\vw,\vu,\vv)\eqbydef\frac12\int_\Omega\vv\trans\GRAD\vu\MVPROD\vw - \frac12\int_\Omega\vu\trans\GRAD\vv\MVPROD\vw,
\end{equation}
where $\trans$ denotes the transpose operator.
  Notice that all the discussion below can be easily adapted to more general boundary conditions, but we stick to the homogeneous Dirichlet case for simplicity of exposition.
Also, the modifications to handle variable kinematic viscosities are briefly discussed in Remark~\ref{rem:var.nu}.

The literature on the numerical approximation of problem \eqref{eq:weak} is vast, and giving a detailed account lies out of the scope of the present work.
We therefore mention here only those numerical methods which share similar features with our approach.
The possibility to increase the approximation order and, possibly, to use general meshes meshes analogous to the ones considered here are supported by discontinuous Galerkin (DG) methods.
Their application to the incompressible Navier--Stokes problem has been considered in several works starting from the early 00's; a non exhaustive list of references includes \cite{Karakashian.Katsaounis:00,Cockburn.Kanschat.ea:05,Girault.Riviere.ea:05,Bassi.Crivellini.ea:06,Bassi.Crivellini.ea:07,Di-Pietro.Ern:10,Bassi.Botti.ea:12*1,Riviere.Sardar:14,Tavelli.Dumbser:14}; cf. also \cite[Chapter 6]{Di-Pietro.Ern:12} for a pedagogical introduction.
The Hybridizable discontinuous Galerkin (HDG) method of \cite{Castillo.Cockburn.ea:00,Cockburn.Gopalakrishnan.ea:09} has also been applied to the discretization of the incompressible Navier--Stokes equations in several recent works \cite{Nguyen.Peraire.ea:11,Cesmelioglu.Cockburn.ea:15,Giorgiani.Fernandez-Mendez.ea:14,Qiu.Shi:16,Ueckermann.Lermusiaux:16}.
Albeit this is not explicitly pointed out in all of the above references, also HDG methods often support general meshes as well as the possibility to increase the approximation order.
The relation between HDG and HHO methods (originally introduced in~\cite{Di-Pietro.Ern:15} in the context of quasi-incompressible linear elasticity) has been recently explored in \cite{Cockburn.Di-Pietro.ea:16} for a pure diffusion problem.
Therein it is shown that, for the same set of globally coupled face-based DOFs, the HHO technology can improve the original orders of convergence of HDG methods while using fewer element-based DOFs for the vector variable and reducing the size of the local computations.
This technology can also be used to derive novel HDG methods with the same favorable features as HHO methods; cf., in particular, \cite[Table 1]{Cockburn.Di-Pietro.ea:16} for further details.

The HHO scheme studied in this work hinges on DOFs located at mesh elements and faces that are discontinuous polynomials of degree $k\ge 0$.
Based on these DOFs, by solving local problems inside each element we obtain reconstructions of the velocity and of its gradient that are used to formulate the diffusive and convective terms in the momentum balance equation, as well as a reconstruction of the divergence used in the velocity-pressure coupling term.
More precisely, the discretization of the viscous term stems from a variation of the diffusive bilinear form originally introduced in \cite{Di-Pietro.Ern.ea:14};
for the convective term, we use a skew-symmetric formulation designed so as not to contribute to the kinetic energy balance in the spirit of the design property \cite[(T1)]{Di-Pietro.Ern:10};
the velocity-pressure coupling is, on the other hand, analogous to that of \cite{Aghili.Boyaval.ea:15,Di-Pietro.Ern.ea:16}.

The resulting method has several advantageous features:
\begin{inparaenum}[(i)]
\item it supports general meshes possibly including polyhedral elements and nonmatching interfaces (resulting, e.g., from nonconforming mesh refinement);
\item it allows one to increase the spatial approximation order to accelerate convergence in the presence of (locally) regular solutions;
\item it is (relatively) inexpensive thanks to the possibility of statically condensing all element-based velocity DOFs and all but one pressure DOF per element at each nonlinear iteration (cf. Remark~\ref{rem:static.cond} below for further details).
\end{inparaenum}
Additionally, thanks to the underlying fully discontinuous polynomial spaces, the proposed method can be expected to accommodate abrupt variations of the unknowns in the vicinity of boundary layers.
Existence of a discrete solution is proved in Theorem~\ref{thm:existence} below for general data resorting to classical arguments in nonlinear analysis \cite{Deimling:85}.
Uniqueness, on the other hand, is shown in Theorem~\ref{thm:uniqueness} below under a standard smallness assumption on the volumetric body force.

A complete convergence analysis of the method is carried out.
First, using a compactness argument inspired by the recent literature on finite volume methods (cf., e.g., \cite{Eymard.Gallouet.ea:00,Eymard.Herbin.ea:07,Eymard.Gallouet.ea:10,Chainais-Hillairet.Krell.ea:15}), we show in Theorem~\ref{thm:conv.min.reg} below that the sequence of discrete solutions on a refined mesh family converges (up to a subsequence) to the continuous one for general data and without assuming more regularity for the exact solution than required by the weak formulation.
Convergence extends to the whole sequence when the continuous solution is unique.
The use of compactness techniques in the context of high-order methods is quite original, and we can only name \cite{Di-Pietro.Ern:10,Di-Pietro.Ern:12} when it comes to the Navier--Stokes problem.
Key technical results required to prove convergence by compactness are the discrete Sobolev embeddings and compactness results recently proved in \cite{Di-Pietro.Droniou:16} in the context of nonlinear Leray--Lions problems.

Then, in Theorem~\ref{thm:err.est} below, we prove error estimates for regular exact solutions under a suitable data smallness assumption.
When polynomials of degree $k\ge 0$ are used, we show that both the energy-norm of the velocity and the $L^2$-norm of the pressure converge as $h^{k+1}$ ($h$ denotes here the meshsize).
These convergence rates are similar to the ones recently derived in \cite{Qiu.Shi:16} for a HDG method with pressure and velocity spaces chosen as in~\cite{Egger.Waluga:13,Wang.Ye:16}.
A major difference with respect to \cite{Qiu.Shi:16} is that we obtain them here using polynomials $k$ instead of $(k+1)$ inside mesh elements (this is precisely one of the major outcomes of the HHO technology identified in \cite{Cockburn.Di-Pietro.ea:16}).
Another difference with respect to \cite{Cesmelioglu.Cockburn.ea:15,Qiu.Shi:16} is that our trilinear form is expressed in terms of a discrete gradient reconstruction and designed so that it does not contribute to the kinetic energy balance, a feature which simplifies several arguments in the analysis; cf. Remark \ref{rem:th.HDG} for further details.
We also show numerically that the $L^2$-norm of the error on the velocity converges as $h^{k+2}$. This result is not surprising, as a similar analysis as the one of \cite{Qiu.Shi:16} can be expected to apply also in our case (the details are postponed to a future work).

The rest of this paper is organized as follows.
In Section~\ref{sec:mesh} we introduce mesh-related notations and recall a few basic results on broken functional spaces.
In Section~\ref{sec:discr} we define the local reconstructions, formulate the discretizations of the various terms appearing in \eqref{eq:weak}, and state the discrete problem.
In Section~\ref{sec:conv} we discuss the existence and uniqueness of a discrete solution, prove convergence to minimal regularity exact solutions for general data, and estimate the convergence rate for smooth exact solutions and small data.
The theoretical results are illustrated on a numerical example including a comparison with a HDG-inspired trilinear form.
In Section~\ref{sec:properties} we give proof of the properties of the viscous and velocity-pressure coupling bilinear forms and of the convective trilinear form used in the analysis.

%------------------------------------------------------------------------------%

\section{Mesh and basic results}\label{sec:mesh}

Let ${\cal H}\subset \Real_*^+ $ denote a countable set of meshsizes having $0$ as its unique accumulation point.
We consider refined mesh sequences $(\Th)_{h \in {\cal H}}$ where, for all $ h \in {\cal H} $, $\Th = \{T\}$ is a finite collection of nonempty disjoint open polyhedral elements such that $\overline{\Omega}=\bigcup_{T\in\Th}\closure{T}$ and $h=\max_{T\in\Th} h_T$
($h_T$ stands for the diameter of $T$).
A hyperplanar closed connected subset $F$ of $\closure{\Omega}$ is called a face if it has nonzero $(d-1)$-dimensional Hausdorff measure and
\begin{inparaenum}[(i)]
\item either there exist $T_1,T_2\in\Th $ such that $F=\partial T_1\cap\partial T_2$ (and $F$ is an interface) or 
\item there exists $T\in\Th$ such that $F=\partial T\cap\partial\Omega$ (and $F$ is a boundary face).
\end{inparaenum}
The set of interfaces is denoted by $\Fhi$, the set of boundary faces by $\Fhb$, and we let
$\Fh\eqbydef\Fhi\cup\Fhb$.
For all $T\in\Th$, the set $\Fh[T]\eqbydef\{F\in\Fh\st F\subset\partial T\}$ collects the faces lying on the boundary of $T$ and, for all $F\in\Fh[T]$, we denote by $\normal_{TF}$ the normal to $F$ pointing out of $T$.
A normal vector $\normal_F$ is associated to each internal face by fixing once and for all an (arbitrary) orientation, whereas for boundary faces $\normal_F$ points out of $\Omega$.

We assume that $(\Th)_{h\in{\cal H}}$ is admissible in the sense of \cite[Chapter~1]{Di-Pietro.Ern:12}, i.e., for all $h\in{\cal H}$, $\Th$ admits a matching simplicial submesh $\fTh$ and there exists a real number $\varrho>0$ (the mesh regularity parameter) independent of $h$ such that the following conditions hold:
\begin{inparaenum}[(i)]
\item for all $h\in{\cal H}$ and every simplex $S\in\fTh$ of diameter $h_S$ and inradius $r_S$, $\varrho h_S\le r_S$;
\item for all $h\in{\cal H}$, all $T\in\Th$, and all $S\in\fTh$ such that $S\subset T$, $\varrho h_T \le h_S$.
  %% \item every mesh element $T\in\Th$ is star-shaped with respect to a ball of radius $\varrho h_T$.
\end{inparaenum}
We refer to \cite[Chapter~1]{Di-Pietro.Ern:12} and \cite{Di-Pietro.Droniou:16,Di-Pietro.Droniou:16*1} for a set of geometric and functional analytic results valid on admissible meshes.
We recall, in particular, that, under these regularity assumptions, the number of faces of each element is uniformly bounded.

Let $X$ be a subset of $\Real^d$ and, for an integer $l\ge 0$, denote by $\Poly{l}(X)$ the space spanned by the restrictions to $X$ of polynomials in the space variables of total degree $l$.
In what follows, the set $X$ will represent a mesh element or face.
We denote by $\lproj[X]{l}:L^1(X)\to\Poly{l}(X)$ the $L^2$-orthogonal projector such that, for all $v\in L^1(X)$,
\begin{equation}\label{eq:lproj}
  \int_X(v-\lproj[X]{l}v)w = 0\qquad\forall w\in\Poly{l}(X).
\end{equation}
The vector- and matrix-valued $L^2$-orthogonal projectors, both denoted by $\vproj[X]{l}$, are obtained applying $\lproj[X]{l}$ component-wise.
The following optimal $W^{s,p}$-approximation properties are proved in \cite[Appendix A.2]{Di-Pietro.Droniou:16} using the classical theory of \cite{Dupont.Scott:80} (cf. also \cite[Chapter 4]{Brenner.Scott:08}):
There is $C>0$ such that, for all $l\ge 0$, all $h\in{\cal H}$, all $T\in\Th$, all $s\in\{1,\ldots,l+1\}$, all $p\in[1,+\infty]$, all $v\in W^{s,p}(T)$, and all $m\in\{0,\ldots,s-1\}$,
\begin{equation}\label{eq:lproj.approx}
  \seminorm[W^{m,p}(T)]{v-\lproj[T]{l}v} + h_T^{\frac1p}\seminorm[{W^{m,p}(\Fh[T]})]{v-\lproj[T]{l}v}\le C h_T^{s-m}\seminorm[W^{s,p}(T)]{v},
\end{equation}
where $W^{m,p}(\Fh[T])$ is spanned by functions that are in $W^{m,p}(F)$ for all $F\in\Fh[T]$.
At the global level, the space of broken polynomial functions on $\Th$ of degree $l$ is denoted by $\Poly{l}(\Th)$, and $\lproj{l}$ is the corresponding $L^2$-orthogonal projector.
The broken gradient operator on $\Th$ is denoted by $\GRADh$.

Let $p\in[1,+\infty]$. We recall the following continuous trace inequality:
There is $C>0$ such that, for all $h\in{\cal H}$ and all $T\in\Th$ it holds for all $v\in W^{1,p}(T)$,
\begin{equation}\label{eq:trace.cont}
  h_T^{\frac1p}\norm[L^p(\partial T)]{v}\le C\left( \norm[L^p(T)]{v} + h_T\norm[L^p(T)^d]{\GRAD v}\right).
\end{equation}
Let an integer $l\ge 0$ be fixed.
Using \eqref{eq:trace.cont} followed by the discrete inverse inequality
\begin{equation}\label{eq:inverse}
  \norm[L^p(T)^d]{\GRAD v}\le Ch_T^{-1}\norm[L^p(T)]{v},
\end{equation}
valid for all $T\in\Th$ and $v\in\Poly{l}(T)$ with $C>0$ independent of $h$ and of $T$, we obtain the following discrete trace inequality:
There is $C>0$ such that, for all $h\in{\cal H}$ and all $T\in\Th$ it holds for all $v\in\Poly{l}(T)$,
\begin{equation}\label{eq:trace}
  h_T^{\frac1p}\norm[L^p(\partial T)]{v}\le C\norm[L^p(T)]{v}.
\end{equation}

Throughout the paper, we often write $a\lesssim b$ (resp. $a\gtrsim b$) to mean $a\le Cb$ (resp. $a\ge Cb$) with real number $C>0$ independent of the meshsize $h$ and of the kinematic viscosity $\nu$.
Constants are named when needed in the discussion.

%------------------------------------------------------------------------------%

\section{Discretization}\label{sec:discr}

In this section we define the discrete counterparts of the various terms appearing in \eqref{eq:weak}, state their properties, and formulate the discrete problem.

\subsection{Discrete spaces}
Let a polynomial degree $k\ge 0$ be fixed.
We define the following hybrid space containing element-based and face-based velocity DOFs:
\begin{equation}\label{eq:Uh}
  \Uh\eqbydef\left(
  \bigtimes_{T\in\Th}\Poly{k}(T)^d
  \right)\times\left(
  \bigtimes_{F\in\Fh}\Poly{k}(F)^d
  \right).
\end{equation}
For the elements of $\Uh$ we use the underlined notation $\uv[h] = \left((\vv_T)_{T\in\Th},(\vv_F)_{F\in\Fh}\right)$.
We define the global interpolator $\Ih:H^1(\Omega)^d\to\Uh$ such that, for all $\vv\in H^1(\Omega)^d$,
$$
\Ih\vv\eqbydef\left(
(\vproj[T]{k}\vv)_{T\in\Th}, (\vproj[F]{k}\vv)_{F\in\Fh}
\right).
$$
For every mesh element $T\in\Th$, we denote by $\UT$ and $\IT$ the restrictions to $T$ of $\Uh$ and $\Ih$, respectively.
Similarly, $\uv=(\vv_T,(\vv_F)_{F\in\Fh[T]})$ denotes the restriction to $T$ of a generic vector $\uv[h]\in\Uh$.
Also, for an element $\uv[h]\in\Uh$, we denote by $\vv_h$ (no underline) the broken polynomial function in $\Poly{k}(\Th)^d$ such that $\restrto{\vv_h}{T}=\vv_T$ for all $T\in\Th$.
Finally, we define on $\Uh$ the following seminorm:
\begin{equation}\label{eq:norm1h}
  \norm[1,h]{\uv[h]}^2\eqbydef\sum_{T\in\Th}\norm[1,T]{\uv[T]}^2,
\end{equation}
where, for all $T\in\Th$,
\begin{equation}\label{eq:norm1T}
  \norm[1,T]{\uv[T]}^2\eqbydef\norm[T]{\GRAD\vv_T}^2 + \seminorm[1,\partial T]{\uv[T]}^2,
  \qquad\seminorm[1,\partial T]{\uv[T]}^2\eqbydef\sum_{F\in\Fh[T]}h_F^{-1}\norm[F]{\vv_F-\vv_T}^2.
\end{equation}
The following boundedness property holds for the global interpolator $\Ih$:
For all $\vv\in H^1(\Omega)^d$,
\begin{equation}\label{eq:Ih.cont}
  \norm[1,h]{\Ih\vv}\le C_I\norm[H^1(\Omega)^d]{\vv},
\end{equation}
with real number $C_I>0$ independent of $h$.

The following velocity and pressure spaces embed the homogeneous boundary conditions for the velocity and the zero-average constraint for the pressure, respectively:
\begin{equation}\label{eq:UhD.Ph}
  \UhD\eqbydef\left\{
  \uv[h]\in\Uh\st\vv_F=\vec{0}\quad\forall F\in\Fhb
  \right\},\qquad
  \Ph\eqbydef\left\{
  q_h\in\Poly{k}(\Th)\;\Big|\; \int_\Omega q_h=0
  \right\}.
\end{equation}

It is a simple matter to check that the map $\norm[1,h]{{\cdot}}$ defines a norm on $\UhD$.
We also note the following discrete Sobolev embeddings, a consequence of \cite[Proposition 5.4]{Di-Pietro.Droniou:16}:
For all $p\in\lbrack 1,+\infty)$ if $d=2$, $p\in[1,6]$ if $d=3$, it holds for all $\uv[h]\in\UhD$,
\begin{equation}\label{eq:sobolev}
  \norm[L^p(\Omega)^d]{\vv_h}\le C_s\norm[1,h]{\uv[h]},
\end{equation}
with real number $C_s>0$ independent of $h$.

\subsection{Reconstructions of differential operators}

Let an element $T\in\Th$ be fixed.
For any polynomial degree $l\ge 0$, we define the local gradient reconstruction operator $\GT[l]:\UT\to\Poly{l}(T)^{d\times d}$ such that, for all $\uv[T]\in\UT$ and all $\mtau\in\Poly{l}(T)^{d\times d}$,
\begin{subequations}
  \begin{align}\label{eq:GT}
    \int_T\GT[l]\uv[T]\SSCAL\mtau
    &=-\int_T\vv_T\SCAL(\DIV\mtau)
    + \sum_{F\in\Fh[T]}\int_F\vv_F\SCAL(\mtau\MVPROD\normal_{TF})
    \\\label{eq:GT'}
    &= \int_T\GRAD\vv_T\SSCAL\mtau
    + \sum_{F\in\Fh[T]}\int_F(\vv_F-\vv_T)\SCAL(\mtau\MVPROD\normal_{TF}),
  \end{align}
\end{subequations}
where we have used integration by parts to pass to the second line.
In \eqref{eq:GT}, the right-hand is designed to resemble an integration by parts formula where the role of the function in volumetric and boundary integrals is played by element-based and face-based DOFs, respectively.

For the discretization of the viscous term, we will need the local velocity reconstruction operator $\pT:\UT\to\Poly{k+1}(T)^d$ obtained in a similar spirit as $\GT$ and such that, for all $\uv[T]\in\UT$,
\begin{equation}\label{eq:pT}
\int_T\GRAD\pT\uv[T]\SSCAL\GRAD\vw = -\int_T\vv_T\SCAL\LAPL\vw + \sum_{F\in\Fh[T]}\int_F\vv_F\SCAL(\GRAD\vw\MVPROD\normal_{TF})
\qquad\forall\vw\in\Poly{k+1}(T)^d,\qquad
\end{equation}
with closure condition $\int_T(\pT\uv[T]-\vv_T)=\vec{0}$.

Finally, we define the discrete divergence operator $\DT:\UT\to\Poly{k}(T)$ such that, for all $\uv[T]\in\UT$ and all $q\in\Poly{k}(T)$,
\begin{subequations}
  \begin{align}\label{eq:DT}
    \int_T\DT\uv[T]q
    &= -\int_T\vv_T\SCAL\GRAD q
    + \sum_{F\in\Fh[T]}\int_F(\vv_F\SCAL\normal_{TF})q
    \\ \label{eq:DT'}
    &= \int_T(\DIV\vv_T) q + \sum_{F\in\Fh[T]}\int_F(\vv_F-\vv_T)\SCAL\normal_{TF} q,
  \end{align}
\end{subequations}
where we have used integration by parts to pass to the second line.
By definition, we have
\begin{equation}\label{eq:DT.trGT}
  \DT={\rm tr}(\GT).
\end{equation}

We also define global versions of the gradient, velocity reconstruction, and divergence operators letting $\Gh[l]:\Uh\to\Poly{l}(\Th)^{d\times d}$, $\ph:\Uh\to\Poly{k+1}(\Th)^d$, and $\Dh:\Uh\to\Poly{k}(\Th)$ be such that, for all $\uv[h]\in\Uh$ and all $T\in\Th$,
$$
\restrto{(\Gh[l]\uv[h])}{T} \eqbydef \GT[l]\uv[T],\qquad
\restrto{(\ph\uv[h])}{T}\eqbydef \pT\uv[T],\qquad
\restrto{(\Dh\uv[h])}{T} \eqbydef \DT\uv[T].
$$

\begin{proposition}[{Properties of $\Gh[l]$}]\label{prop:Gh.properties}
  The global discrete gradient operator $\Gh[l]$ satisfies the following properties:
  \begin{enumerate}[1)]
  \item \emph{Boundedness.} For all $l\ge 0$ and all $\uv[h]\in\Uh$, it holds
    \begin{equation}\label{eq:Gh.cont}
      \norm[L^2(\Omega)^{d\times d}]{\Gh[l]\uv[h]}\lesssim\norm[1,h]{\uv[h]}.  
    \end{equation}
  \item \emph{Consistency.} For all $l\ge 0$ and all $\vv\in H^{m}(\Omega)^d$ with $m=l+2$ if $l\le k$, $m=k+1$ otherwise,
    \begin{equation}\label{eq:Gh.approx}
      \norm[L^2(\Omega)^{d\times d}]{\Gh[l]\Ih\vv-\GRAD\vv}
      + \left(\sum_{T\in\Th} h_T \norm[L^2(\partial T)^{d\times d}]{\GT[l]\IT\vv-\GRAD\vv}^2\right)^{\frac12}
      \lesssim h^{m-1}\norm[H^{m}(\Omega)^d]{\vv}.
    \end{equation}
    As a consequence, for all $\vPhi\in C_c^\infty(\Omega)^d$ and all $l,k$ such that $m>1$ (i.e., provided $l=0$ if $k=0$), $\Gh[l]\Ih\vPhi\to\GRAD\vPhi$ strongly in $L^2(\Omega)^{d\times d}$.
  \item \emph{Sequential consistency.} Let $(\uv[h])_{h\in{\cal H}}$ denote a sequence in $(\UhD)_{h\in{\cal H}}$ bounded in the $\norm[1,h]{{\cdot}}$-norm.
  Then, there is $\vv\in\U$ such that
  \begin{itemize}
  \item $\vv_h\to\vv$ strongly in $L^p(\Omega)^d$ for all $p\in[1,+\infty)$ if $d=2$, $p\in\lbrack 1,6)$ if $d=3$;
  \item $\Gh[l]\uv[h]\rightharpoonup\GRAD\vv$ weakly in $L^2(\Omega)^{d\times d}$ for all $l\ge 0$.
  \end{itemize}
  \end{enumerate}
\end{proposition}

\begin{proof}
  \begin{asparaenum}[1)]
  \item \emph{Boundedness.} Let an element $T\in\Th$ be fixed, make $\mtau=\GT[l]\uv[T]$ in \eqref{eq:GT'} and use the Cauchy--Schwarz inequality followed by the discrete trace inequality \eqref{eq:trace} with $p=2$ to infer, for all $\uv[T]\in\UT$,
    $$
    \norm[L^2(T)^{d\times d}]{\GT[l]\uv[T]}\lesssim\norm[1,T]{\uv[T]}.
    $$
    Squaring the above inequality and summing over $T\in\Th$, \eqref{eq:Gh.cont} follows.
    
  \item \emph{Consistency.} Let $\vv\in H^{m}(\Omega)^d$.
    For all $T\in\Th$, plugging the definition of $\IT\vv$ into \eqref{eq:GT}, we get, for all $\mtau\in\Poly{l}(T)^{d\times d}$,
    \begin{equation}\label{eq:GT.consistency.basic}
      \int_T(\GT[l]\IT\vv-\GRAD\vv)\SSCAL\mtau = -\int_T(\vproj[T]{k}\vv-\vv)\SCAL(\DIV\mtau) + \sum_{F\in\Fh[T]}\int_F(\vproj[F]{k}\vv-\vv)\SCAL(\mtau\MVPROD\normal_{TF}).
    \end{equation}
    Recalling the definition \eqref{eq:lproj} of $\vproj[T]{k}$ and $\vproj[F]{k}$, we get from the previous expression that
    \begin{equation}\label{eq:GT.euler}
      \forall n\le k,\qquad    
      \int_T(\GT[l]\IT\vv-\GRAD\vv)\SSCAL\mtau = 0
      \qquad\forall\mtau\in\Poly{n}(T)^{d\times d},
    \end{equation}
    since $(\DIV\mtau)\in\Poly{n-1}(T)^d\subset\Poly{k}(T)^d$ and $\restrto{\mtau}{F}\normal_{TF}\in\Poly{n}(F)^d\subset\Poly{k}(F)^d$.
    If $l\le k$, this shows in particular that for all $T\in\Th$ it holds
    \begin{equation}\label{eq:GT.commuting}
      \GT[l]\IT\vv=\mproj[T]{l}\GRAD\vv,
    \end{equation}
    and \eqref{eq:Gh.approx} is an immediate consequence of the approximation properties \eqref{eq:lproj.approx} of the $L^2$-orthogonal projector.
    On the other hand, if $l>k$, making $\mtau=\GT[l]\IT\vv-\mproj{l}\GRAD\vv$ in \eqref{eq:GT.consistency.basic} and using the Cauchy--Schwarz, discrete inverse \eqref{eq:inverse} and trace \eqref{eq:trace} inequalities (both with $p=2$) to bound the right-hand side, we infer $\norm[L^2(T)^{d\times d}]{\GT[l]\IT\vv-\mproj{l}\GRAD\vv}\lesssim h^k\norm[H^{k+1}(T)^d]{\vv}$.
    Hence, using the triangle inequality, we arrive at
    $$
    \norm[L^2(T)^{d\times d}]{\GT[l]\IT\vv-\GRAD\vv}\lesssim h_T^k\norm[H^{k+1}(T)^d]{\vv},
    $$
    and \eqref{eq:Gh.approx} follows squaring and summing over $T\in\Th$.
    
  \item \emph{Sequential consistency.} The proof for $l=k$ in the scalar case is given in \cite[Proposition 5.6]{Di-Pietro.Droniou:16}. A close inspection shows that the arguments still stand when $l\neq k$ provided that we replace $\vproj[T]{k}\vPhi$ by $\vproj[T]{0}\vPhi$.
    %% To this purpose, we take $\vPhi\in C^\infty(\Real^d)^{d\times d}$ and observe that
    %%   \begin{align*}
    %%     \int_{\Omega}\Gh[l]\uv[h]\SSCAL\vPhi
    %%     &=\sum_{T\in\Th}\int_T(\GT[l]\uv-\GRAD\vv_T)\SSCAL(\vPhi-\vproj[T]{0}\vPhi)
    %%     + \sum_{T\in\Th}\int_T(\GT[l]\uv-\GRAD v_T)\SSCAL\vproj[T]{0}\vPhi
    %%     + \sum_{T\in\Th}\int_T\GRAD\vv_T\SSCAL\vPhi
    %%     \\
    %%     &\eqbydef\term_1
    %%     +\sum_{T\in\Th}\sum_{F\in\Fh[T]}\int_F(\vv_F- \vv_T)\SCAL\vproj[T]{0}(\vPhi\MVPROD\normal_{TF})
    %%     +\sum_{T\in\Th}\int_T\GRAD\vv_T\SSCAL\vPhi
    %%     \\
    %%     &=\term_1
    %%     +\sum_{T\in\Th}\sum_{F\in\Fh[T]}\int_F(\vv_F- \vv_T)\SCAL\left[
    %%       (\vproj[T]{0}\vPhi-\vPhi)\MVPROD\normal_{TF}
    %%       \right]
    %%       -\sum_{T\in\Th}\int_T\vv_T\SCAL\DIV\vPhi
    %%     \\
    %%     &\eqbydef\term_1 + \term_2 - \sum_{T\in\Th}\int_T\vv_T\SCAL\DIV\vPhi.
    %%   \end{align*}
    \qedhere
  \end{asparaenum}
\end{proof}

\begin{remark}[Commuting property for $\Dh$]\label{rem:Dh.commuting}
  Combining \eqref{eq:GT.commuting} with~\eqref{eq:DT.trGT}, it is a simple matter to infer the following commuting property for $\Dh$: For all $\vv\in\U$,
  \begin{equation}\label{eq:Dh.commuting}
    \Dh\Ih\vv = \lproj{k}(\DIV\vv).
  \end{equation}
  This property is crucial to prove the inf-sup condition of Proposition \ref{prop:bh.properties} below using classical arguments in the analysis of saddle-point problems (cf., e.g., the reference textbook~\cite{Boffi.Brezzi.ea:13}).
\end{remark}

\subsection{Viscous term}

The viscous term is discretized by means of the bilinear form $a_h$ such that, for all $\uu[h],\uv[h]\in\Uh$,
\begin{equation}\label{eq:ah}
  a_h(\uu[h],\uv[h])
  \eqbydef\int_\Omega\Gh[k]\uu[h]\SSCAL\Gh[k]\uv[h] + s_h(\uu[h],\uv[h]),
\end{equation}
with stabilization bilinear form $s_h$ defined as follows:
$$
s_h(\uu[h],\uv[h])\eqbydef\sum_{T\in\Th}\sum_{F\in\Fh[T]}h_F^{-1}\int_F\dTF\uu[T]\SCAL\dTF\uv[T],
$$
where, for all $T\in\Th$ and all $F\in\Fh[T]$, we have introduced the face-based residual operator $\dTF:\UT\to\Poly{k}(F)^d$ such that
\begin{equation}\label{eq:dTF}
  \dTF\uv[T]\eqbydef\vproj[F]{k}\left(\vv_F - \pT\uv[T] - \vproj[T]{k}(\vv_T - \pT\uv[T])\right).
\end{equation}
This specific form of the penalized residual ensures the following consistency property (cf. \cite[Remark 6]{Di-Pietro.Ern.ea:14} for further insight):
For all $\vv\in H^{k+2}(\Omega)^d$,
\begin{equation}\label{eq:sh.consistency}
  s_h(\Ih\vv,\Ih\vv)^{\nicefrac12}\lesssim h^{k+2}\norm[H^{k+2}(\Omega)]{\vv}.
\end{equation}
The proof of the following result is postponed to Section \ref{sec:ah.properties}.
\begin{proposition}[Properties of $a_h$]\label{prop:ah.properties}
  The bilinear form $a_h$ has the following properties:
  \begin{enumerate}[1)]
  \item \emph{Stability and boundedness.} It holds, for all $\uv[h]\in\Uh$,
    \begin{equation}\label{eq:ah.stab.cont}
      C_a^{-1}\norm[1,h]{\uv[h]}^2\le\norm[a,h]{\uv[h]}^2\eqbydef a_h(\uv[h],\uv[h])\le C_a\norm[1,h]{\uv[h]}^2,
    \end{equation}
    with real number $C_a>0$ independent of $h$.
    Consequently, the map $\norm[a,h]{{\cdot}}$ defines a norm on $\UhD$ uniformly equivalent to $\norm[1,h]{{\cdot}}$.\label{item:ah.stab.cont}
  \item \emph{Consistency.} For all $\vv\in \U\cap H^{k+2}(\Omega)^d$, it holds
    \begin{equation}\label{eq:ah.consistency}
      \sup_{\uw[h]\in\UhD,\norm[1,h]{\uw[h]}=1} \left|
      \int_\Omega\LAPL\vv\SCAL\vw_h + a_h(\Ih\vv,\uw[h])
      \right|\lesssim h^{k+1}\norm[H^{k+2}(\Omega)^d]{\vv}.
    \end{equation} \label{item:ah.consistency}
  \item \emph{Sequential consistency.} Let $(\uv[h])_{h\in{\cal H}}$ denote a sequence in $(\UhD)_{h\in{\cal H}}$ bounded in the $\norm[1,h]{{\cdot}}$-norm with limit $\vv\in \U$ (cf. point 3) in Proposition \ref{prop:Gh.properties}).
    Then, it holds for all $\vPhi\in C_{\rm c}^\infty(\Omega)^d$
    $$
    a_h(\uv[h],\Ih\vPhi)\to a(\vv,\vPhi).
    $$\label{item:ah.seqconsistency}
  \end{enumerate}
\end{proposition}

Some remarks are of order.
\begin{remark}[Alternative viscous bilinear form]\label{rem:ah}
  An alternative choice corresponding to the original HHO bilinear form of \cite{Di-Pietro.Ern.ea:14} is
  $$
  a_h(\uu[h],\uv[h])\eqbydef\int_\Omega\GRADh\ph\uu[h]\SSCAL\GRADh\ph\uv[h] + s_h(\uu[h],\uv[h]),
  $$
  where the difference with respect to~\eqref{eq:ah} lies in the fact that $\GRADh\ph$ replaces $\Gh$ in the consistency term.
  Properties 1)--2) in Proposition~\ref{prop:ah.properties} are straightforward consequences of \cite[Lemma~4 and Theorem~8]{Di-Pietro.Ern.ea:14}, respectively.
  Property 3), on the other hand, would require proving for $\GRADh\ph$ sequential consistency as in point 3) of Proposition~\ref{prop:Gh.properties}.
\end{remark}%

\begin{remark}[Variable kinematic viscosity]\label{rem:var.nu}
  A more general form for the viscous term in~\eqref{eq:strong:momentum} accomodating variable kinematic viscosities $\nu:\Omega\to\Real$ is
  $$
  -\DIV\matr{\sigma}(\vu),\qquad
  \matr{\sigma}(\vu) = 2\nu\GRADs\vec{u},
  $$
  where $\GRADs$ denotes the symmetric gradient operator.
  Our discretization can be modified to accomodate this case adapting the ideas developed in~\cite{Di-Pietro.Ern:15} in the framework of linear elasticity.
  Assume, for the sake of simplicity, that $\nu$ is piecewise constant on a partition of $\Omega$, and that for all $h\in{\cal H}$ the mesh $\Th$ is compliant with the partition (so that jumps of $\nu$ only occur at interfaces).
  For all $T\in\Th$, we define the discrete symmetric gradient operator $\GsT\eqbydef\frac12\big(\GT+(\GT)\trans\big)$ (with $\GT$ defined by \eqref{eq:GT}) and we use instead of~\eqref{eq:pT} the velocity reconstruction such that, for all $\uv[T]\in\UT$,
  \begin{subequations}\label{eq:pT.nu}
    \begin{equation}\label{eq:pT.nu:1}
      \int_T\GRADs\pT\uv[T]\SSCAL\GRADs\vw=
      -\int_T\vv_T\SCAL\DIV(\GRADs\vw) + \sum_{F\in\Fh[T]}\int_F\vv_F\SCAL(\GRADs\vw\MVPROD\normal_{TF})
      \qquad\forall\vw\in\Poly{k+1}(T)^d
    \end{equation}
  and
  \begin{equation}\label{eq:pT.nu:2}
  \int_T(\pT\uv[T]-\vv_T)=\vec{0},\qquad
  \int_T\GRADs\pT\uv[T]=\frac12\sum_{F\in\Fh[T]}\int_F\left(\normal_{TF}\otimes\vv_F-\vv_F\otimes\normal_{TF}\right).
  \end{equation}
  \end{subequations}
  Letting $\Gsh:\Uh\to\Poly{k+1}(\Th)^{d\times d}$ be such that, for all $\uv[h]\in\Uh$, $\restrto{(\Gsh\uv[h])}{T}=\GsT\uv[T]$, the viscous term in~\eqref{eq:discrete:momentum} below is discretized by means of the bilinear form
  $$
  a_{\nu,h}(\uu[h],\uv[h])\eqbydef\int_\Omega2\nu\Gsh\uu[h]:\Gsh\uv[h] + s_{\nu,h}(\uu[h],\uv[h]),
  $$
  with stabilization bilinear form
  $$
  s_{\nu,h}(\uu[h],\uv[h])\eqbydef\sum_{T\in\Th}\sum_{F\in\Fh[T]}\frac{2\nu_T}{h_F}\int_F\dTF\uu[T]\SCAL\dTF\uv[T],
  $$
  where $\nu_T\eqbydef\restrto{\nu}{T}\in\Poly{0}(T)$ and $\dTF$ is formally defined as in \eqref{eq:dTF} but using the velocity reconstruction operator defined by \eqref{eq:pT.nu}.
  In the analysis, the main difference with respect to constant kinematic viscosities is that the polynomial degree $k$ should be taken $\ge 1$ in order to ensure coercivity by a discrete Korn inequality (cf. in particular~\cite[Lemma~4]{Di-Pietro.Ern:15} for insight into this point).
\end{remark}

\subsection{Convective term}
For the discretization of the convective term, we consider here the following trilinear form $\Uh\times\Uh\times\Uh$ expressed in terms of the discrete gradient operator $\Gh[2k]$:
\begin{equation}\label{eq:th}
  t_h(\uw[h],\uu[h],\uv[h])
  \eqbydef\frac12\int_\Omega\vv_h\trans\Gh[2k]\uu[h]\MVPROD\vw_h - \frac12\int_\Omega\vu_h\trans\Gh[2k]\uv[h]\MVPROD\vw_h.
\end{equation}
This expression mimicks the continuous one given in~\eqref{eq:a.b.t} with $\Gh[2k]$ replacing the continuous gradient operator.
  Notice that, in the practical implementation, one does not need to actually compute $\Gh[2k]$ to evaluate $t_h$.
Instead, the following expression can be used, obtained by applying \eqref{eq:GT} twice to expand the terms involving $\Gh[2k]$:
$$
t_h(\uw[h],\uu[h],\uv[h])
= \sum_{T\in\Th} t_T(\uw[T],\uu[T],\uv[T]),
$$
where, for all $T\in\Th$,
\begin{equation}\label{eq:tT}
  \begin{aligned}
    t_T(\uw[T],\uu[T],\uv[T])
    &\eqbydef
    -\frac12\int_T\vu_T\trans\GRAD\vv_T\MVPROD\vw_T + \frac12\int_T\vv_T\trans\GRAD\vu_T\MVPROD\vw_T
    \\
    &\qquad
    + \frac12\sum_{F\in\Fh[T]}\int_F (\vu_F\SCAL\vv_T)(\vw_T\SCAL\normal_{TF})
    - \frac12\sum_{F\in\Fh[T]}\int_F (\vv_F\SCAL\vu_T)(\vw_T\SCAL\normal_{TF}).
  \end{aligned}
\end{equation}

The proof of the following result is postponed to Section \ref{sec:th.properties}.
\begin{proposition}[Properties of $t_h$]\label{prop:th.properties}
  The trilinear form $t_h$ has the following properties:
  \begin{enumerate}[1)]
  \item \emph{Skew-symmetry.} For all $\uv[h],\uw[h]\in\UhD$, it holds
    \begin{equation}\label{eq:th.ssymm}
      t_h(\uw[h],\uv[h],\uv[h])=0.
    \end{equation}
  \item \emph{Boundedness.} For all $\uu[h],\uv[h],\uw[h]\in\UhD$, it holds
    \begin{equation}\label{eq:th.cont}
      |t_h(\uw[h],\uu[h],\uv[h])|\le C_t\norm[1,h]{\uw[h]}\norm[1,h]{\uu[h]}\norm[1,h]{\uv[h]},
    \end{equation}
    with real number $C_t>0$ independent of $h$.
  \item \emph{Consistency.} For all $\vv\in\U\cap H^{k+2}(\Omega)^d$ such that $\DIV\vv = 0$, it holds
    \begin{equation}\label{eq:th.consistency}
      \sup_{\uw[h]\in\UhD,\norm[1,h]{\uw[h]}=1}\left|
      \int_\Omega\vw_h\trans\GRAD\vv~\vv
      - t_h(\Ih\vv,\Ih\vv,\uw[h])
      \right|
      \lesssim h^{k+1}\norm[H^2(\Omega)^d]{\vv}\norm[H^{k+2}(\Omega)^d]{\vv}.
    \end{equation}
  \item \emph{Sequential consistency.} Let $(\uv[h])_{h\in{\cal H}}$ denote a sequence in $(\UhD)_{h\in{\cal H}}$ bounded in the $\norm[1,h]{{\cdot}}$-norm with limit $\vv\in \U$ (cf. point 3) in Proposition \ref{prop:Gh.properties}). Then, for all $\vPhi\in C_{\rm c}^\infty(\Omega)^d$ it holds
    \begin{equation}\label{eq:th.bis}
      t_h(\uv[h],\uv[h],\Ih\vPhi)\to t(\vv,\vv,\vPhi).
    \end{equation}
  \end{enumerate}
\end{proposition}

Some remarks are of order.
  \begin{remark}[Design guidelines]
    The trilinear form $t_h$ appears in the analysis carried out in Section~\ref{sec:conv} only through its properties detailed in Proposition~\ref{prop:th.properties}, with the sole exception of {\bf Step 4} in the proof of Theorem~\ref{thm:conv.min.reg} (strong convergence of the pressure), which requires a more intimate use of its expression.
    Such properties can therefore be intended as design guidelines.
\end{remark}

\begin{remark}[Comparison with a HDG trilinear form]\label{rem:th.HDG}
  A trilinear form inspired by the recent HDG literature is
  \begin{equation}\label{eq:th.HDG}
    t_h^{\rm HDG}(\uw[h],\uu[h],\uv[h])\eqbydef \sum_{T\in\Th} t_T^{\rm HDG}(\uw,\uu,\uv),
  \end{equation}
  where, for all $T\in\Th$,
  $$
  \begin{aligned}
    t_T^{\rm HDG}(\uw,\uu,\uv)
    &\eqbydef
    - \frac12\int_T\vu_T\trans\GRAD\vv_T\MVPROD\vw_T + \frac12\int_T\vv_T\trans\GRAD\vu_T\MVPROD\vw_T
    \\
    &\qquad
    + \frac12\sum_{F\in\Fh[T]}\int_F (\vu_F\SCAL\vv_T)(\vw_F\SCAL\normal_{TF})
    - \frac12\sum_{F\in\Fh[T]}\int_F (\vv_F\SCAL\vu_T)(\vw_F\SCAL\normal_{TF})
    \\
    &\qquad      
    + \frac{\eta}{2}\sum_{F\in\Fh[T]}\int_F (\vu_F-\vu_T)\SCAL(\vv_F-\vv_T)|\vw_F\SCAL\normal_{TF}|.
  \end{aligned}
  $$
  This trilinear form has been recently proposed in \cite{Qiu.Shi:16} (cf. Definition~3.3 therein and also~\cite{Cesmelioglu.Cockburn.ea:15}), where a HDG method is considered with element-based DOFs that are polynomials of degree $(k+1)$ (recall that here we use polynomials of degree $k$, cf.~\eqref{eq:Uh}) and the viscous term is discretized as in \cite{Lehrenfeld:10,Oikawa:15} in order to improve the convergence rates to match the ones of HHO methods (cf. \cite{Cockburn.Di-Pietro.ea:16} for further details, in particular Remark~2.2).
  Comparing the above expression of $t^{\rm HDG}_T$ with \eqref{eq:tT}, we observe the following differences:
  \begin{inparaenum}[(i)]
  \item $\vw_F$ replaces $\vw_T$ in both terms in the second line;
  \item a nonnegative stabilization corresponding to the term in the third line is added, including an user-dependent parameter $\eta\ge 0$ (taken equal to 1 in \cite{Qiu.Shi:16}).
  \end{inparaenum}
  Our analysis can be adapted to this trilinear form.
  In particular, all the properties listed in Proposition~\ref{prop:th.properties} hold for $t_h^{\rm HDG}$ with $\eta=0$.
\end{remark}

\subsection{Pressure-velocity coupling}

The pressure-velocity coupling is realized by means of the bilinear form $b_h$ on $\Uh\times\Ph$ such that, for all $(\uv[h],q_h)\in\Uh\times\Ph$,
\begin{equation}\label{eq:bh}
  b_h(\uv[h],q_h)\eqbydef -\int_\Omega\Dh\uv[h]q_h.
\end{equation}
The proof of the following result is postponed to Section \ref{sec:bh.properties}.
\begin{proposition}[Properties of $b_h$]\label{prop:bh.properties}
  The bilinear form $b_h$ has the following properties:
  \begin{enumerate}[1)]
  %% \item \emph{Boundedness.} For all $(\uv[h],q_h)\in\Uh\times\Ph$ it holds
  %%   \begin{equation}\label{eq:bh.cont}
  %%     |b_h(\uv[h],q_h)|\lesssim\norm[1,h]{\uv[h]}\norm[L^2(\Omega)]{q_h}.
  %%   \end{equation}
  \item \emph{Inf-sup stability.} For all $q_h\in\Ph$ (with $\Ph$ defined by~\eqref{eq:UhD.Ph}), it holds
    \begin{equation}\label{eq:bh.inf-sup}
      \norm[L^2(\Omega)]{q_h}\lesssim\sup_{\uv[h]\in\UhD,\norm[1,h]{\uv[h]}=1} b_h(\uv[h],q_h).
    \end{equation}\label{item:bh.inf-sup}
  \item \emph{Consistency.} For all $q\in H^{k+1}(\Omega)$ it holds
    \begin{equation}\label{eq:bh.consistency}
      \sup_{\uv[h]\in\UhD,\norm[1,h]{\uv[h]}=1}\left|
      \int_\Omega\GRAD q\SCAL\vv_h - b_h(\uv[h],\lproj{k}q)
      \right|\lesssim h^{k+1}\norm[H^{k+1}(\Omega)]{q}.
    \end{equation}\label{item:bh.consistency}
  \item \emph{Sequential consistency.} We have sequential consistency for $b_h$ in the following sense:
    \begin{itemize}
    \item Let $(q_h)_{h\in{\cal H}}$ denote a sequence in $(\Ph)_{h\in{\cal H}}$ bounded in the $L^2(\Omega)$-norm and weakly converging to $q\in P$. Then, for all $\vPhi\in C_c^\infty(\Omega)^d$ it holds
      \begin{equation}\label{eq:bh.seq.cons:1}
        b_h(\Ih\vPhi,q_h)\to b(\vPhi,q).
      \end{equation}
    \item Let $(\uv[h])_{h\in{\cal H}}$ denote a sequence in $(\UhD)_{h\in{\cal H}}$ bounded in the $\norm[1,h]{{\cdot}}$-norm with limit $\vv\in \U$ (cf. point 3)
    in Proposition \ref{prop:Gh.properties}). Then, for all $\varphi\in C_c^\infty(\Omega)$ it holds
      \begin{equation}\label{eq:bh.seq.cons:2}
        b_h(\uv[h],\lproj{k}\varphi)\to b(\vv,\varphi).
      \end{equation}
    \end{itemize}
    \label{item:bh.seq.cons}
  \end{enumerate}
\end{proposition}

\subsection{Discrete problem}

The discrete problem reads: Find $(\uu[h],p_h)\in\UhD\times\Ph$ such that
\begin{subequations}\label{eq:discrete}
  \begin{alignat}{2}
    \label{eq:discrete:momentum}
    \nu a_h(\uu[h],\uv[h]) + t_h(\uu[h],\uu[h],\uv[h]) + b_h(\uv[h],p_h) &= \int_\Omega\vf\SCAL\vv_h &\qquad& \forall\uv[h]\in\UhD,
    \\
    \label{eq:discrete:mass}
    -b_h(\uu[h],q_h) &= 0 &\qquad& \forall q_h\in\Ph.
  \end{alignat}
\end{subequations}

\begin{remark}[Efficient implementation]\label{rem:static.cond}
  When solving the system of nonlinear algebraic equations corresponding to \eqref{eq:discrete} by a first-order (Newton-like) algorithm, all element-based velocity DOFs and all but one pressure DOF per element can be locally eliminated at each iteration by computing the corresponding Schur complement element-wise.
  As all the computations are local, this static condensation procedure is a trivially parallel task which can fully benefit from multi-thread and multi-processor architectures.
  For the details, we refer to \cite[Section 6.2]{Di-Pietro.Ern.ea:16}, where the Stokes problem is considered (the only variation here is that also the linearized convective term appears in the matrices therein denoted by $A_T$).
  As a result, after the elimination to boundary DOFs corresponding to Dirichlet boundary conditions, we end up solving at each iteration a linear system of size
  $$
  d\card{\Fhi}{k+d-1\choose d-1} + \card{\Th}.
  $$
\end{remark}

%------------------------------------------------------------------------------%

\section{Analysis of the method}\label{sec:conv}

In this section we study the existence and uniqueness of the solution to the HHO scheme~\eqref{eq:discrete}, prove convergence to the exact solution for general data, and derive convergence rates under a standard data smallness assumption.

\subsection{Existence and uniqueness}\label{sec:ex.uniq}

The existence of a solution to problem \eqref{eq:discrete} can be proved using the following topological degree lemma (cf., e.g., \cite{Deimling:85}), as originally proposed in \cite{Eymard.Gallouet.ea:98} 
in the context of finite volumes for nonlinear hyperbolic problems; see also \cite{Eymard.Herbin.ea:07,Di-Pietro.Ern:10} for the Navier--Stokes equations.

\begin{lemma}[Topological degree]\label{lem:top.degree}
Let $W$ be a finite-dimensional functional space equipped with a norm $\norm[W]{{\cdot}}$, and let the function $\Psi:W\times[0,1]\to W$ satisfy the following assumptions:
\begin{enumerate}[1)]
\item $\Psi$ is continuous;
\item There exists $\mu>0$ such that, for any $(w,\rho)\in W\times [0,1]$, $\Psi(w,\rho)=0$ implies $\norm[W]{w}\neq\mu$;
\item $\Psi(\cdot,0)$ is an affine function and the equation $\Psi(w,0)=0$ has a solution $w\in W$ such that $\norm[W]{w}<\mu$.
\end{enumerate}
Then, there exists $w\in V$ such that $\Psi(w,1)=0$ and $\norm[W]{w}<\mu$.
\end{lemma}

\begin{theorem}[Existence and a priori bounds]\label{thm:existence}
 There exists a solution  $(\uu[h],p_h)\in\UhD\times\Ph$ to \eqref{eq:discrete}, which satisfies the a priori bounds
  \begin{equation}\label{eq:a-priori}
    \norm[1,h]{\uu[h]}\le C_aC_s\nu^{-1}\norm[L^2(\Omega)^d]{\vf},\qquad
    \norm[L^2(\Omega)]{p_h}\le C\left(\norm[L^2(\Omega)^d]{\vf} + \nu^{-2}\norm[L^2(\Omega)^d]{\vf}^2\right),
  \end{equation}
  with $C_a$ and $C_s$ as in \eqref{eq:ah.stab.cont} and \eqref{eq:sobolev}, respectively, and $C>0$ real number independent of both $h$ and $\nu$.
\end{theorem}

\begin{proof}
 We consider the finite-dimensional space $\Wh\eqbydef\UhD\times\Ph$ equipped with the norm 
 $$
 \norm[W,h]{(\uw[h],r_h)}\eqbydef\norm[1,h]{ \uw[h]}+ \nu^{-1}\norm[L^2(\Omega)]{r_h }
 $$
 and the function $\Psi:\Wh\times[0,1]\to\Wh$ such that, for given $(\uw[h],r_h)\in\Wh$ and $\rho\in[0,1]$, $(\uxi[h],\zeta_h)=\Psi((\uw[h],r_h),\rho)$ is defined as the unique solution of
 \begin{subequations}\label{eq:rho:discrete}
  \begin{alignat}{2}
    \label{eq:rho:discrete:momentum}
  (\uxi[h],\uv[h])_{0,h}&=  \nu a_h(\uw[h],\uv[h]) + \rho t_h(\uw[h],\uw[h],\uv[h]) + b_h(\uv[h],r_h) -\int_\Omega\vf\SCAL\vv_h &\qquad& \forall\uv[h]\in\UhD,
    \\ \label{eq:rho:discrete:mass}
   \int_\Omega\zeta_h q_h&= -b_h(\uw[h],q_h)  &\qquad& \forall q_h\in\Ph,
   \end{alignat}
 \end{subequations}
 where $(.,.)_{0,h}$ is the $L^2$-like scalar product on $\Wh$ defined by
 $$	
 (\uw[h],\uv[h])_{0,h}
 \eqbydef\int_\Omega\vw_h\SCAL\vv_h
 + \sum_{T\in\Th}\sum_{F\in\Fh[T]}h_F\int_F(\vw_F-\vw_T)\SCAL(\vv_F-\vv_T).
 $$
 We next check the assumptions of the topological degree lemma.
 
 \begin{enumerate}[1)]
  \item Since $\Wh $ is a finite-dimensional space, the bilinear forms $a_h$ and $b_h$, the trilinear form $t_h$, and the scalar products are continuous, and so is the case for the function $\Psi$.

  \item Let $(\uw[h],r_h)\in \Wh$ be such that $\Psi((\uw[h],r_h),\rho)=(\underline{\vec{0}},0)$ for some $\rho\in[0,1]$. We next show that
$$
\norm[W,h]{(\uw[h],r_h)}\le (C_aC_s + C)\nu^{-1}\norm[L^2(\Omega)^d]{\vf} + C\nu^{-3}\norm[L^2(\Omega)^d]{\vf}^2,
$$
and point 2) in Lemma \ref{lem:top.degree} is verified for $$\mu=(C_aC_s + C)\nu^{-1}\norm[L^2(\Omega)^d]{\vf} + C\nu^{-3}\norm[L^2(\Omega)^d]{\vf}^2 + \epsilon$$ with $\epsilon>0$.
Recalling the coercivity of $a_h$ expressed by the first inequality in \eqref{eq:ah.stab.cont}, 
making $\uv[h]=\uw[h]$ in \eqref{eq:rho:discrete:momentum} and observing that $t_h(\uw[h],\uw[h],\uw[h])=0$ owing to skew-symmetry \eqref{eq:th.ssymm}
and that $b_h(\uw[h],r_h)=0$ owing to \eqref{eq:rho:discrete:mass} with $q_h=r_h$, we have
  $$
  \nu C_a^{-1}\norm[1,h]{\uw[h]}^2\le \nu\norm[a,h]{\uw[h]}^2 = \int_\Omega\vf\SCAL\vw_h
  \le\norm[L^2(\Omega)^d]{\vf}\norm[L^2(\Omega)^d]{\vw_h}
  \le C_s\norm[L^2(\Omega)^d]{\vf}\norm[1,h]{\uw[h]},
 $$
  where we have used the discrete Poincar\'{e} inequality \eqref{eq:sobolev} with $p=2$ to conclude. The bound on $\uw[h]$ follows.
  To prove the bound on $r_h$, we proceed as follows:
  $$
  \begin{aligned}
    \norm[L^2(\Omega)]{r_h}
    &\lesssim\sup_{\uv[h]\in\UhD,\norm[1,h]{\uv[h]}=1} b_h(\uv[h],r_h)
    \\
    &=\sup_{\uv[h]\in\UhD,\norm[1,h]{\uv[h]}=1}\left(
    \int_\Omega\vf\SCAL\vv_h - \nu a_h(\uw[h],\uv[h]) - \rho t_h(\uw[h],\uw[h],\uv[h])    
    \right)
    \\
    &\lesssim\norm[L^2(\Omega)^d]{\vf} + \nu\norm[1,h]{\uw[h]} + \rho \norm[1,h]{\uw[h]}^2
    \lesssim\norm[L^2(\Omega)^d]{\vf} + \nu^{-2}\norm[L^2(\Omega)^d]{\vf}^2,
  \end{aligned}
  $$
  where we have used the inf-sup condition \eqref{eq:bh.inf-sup} on $b_h$ in the first line and~\eqref{eq:rho:discrete:momentum} to pass to the second line;
  the Cauchy--Schwarz and the discrete Poincar\'{e} inequalities together with the boundedness of $a_h$ and $t_h$ expressed by the second inequality in \eqref{eq:ah.stab.cont} and by \eqref{eq:th.cont}, respectively, are used to pass to the third line;
  the bound on the velocity and the fact that $\rho\le 1$ allow to conclude.

\item $\Psi(\cdot,0)$ is an affine function from $\Wh$ to $\Wh$. The fact that $\Psi(\cdot,0)$ is invertible corresponds to the well-posedness of the HHO scheme for the Stokes problem, and can therefore be proved using the arguments of \cite[Lemma 3]{Di-Pietro.Ern.ea:16} (which classically rely on the coercivity  of $a_h$ expressed by the first inequality in~\eqref{eq:ah.stab.cont} and the inf-sup condition~\eqref{eq:bh.inf-sup} for $b_h$).
  Additionally, the unique solution $(\uw[h],r_h)\in\Wh$ to the equation $\Psi((\uw[h],r_h),0)=0$ satisfies $\norm[W,h]{(\uw[h],r_h)}<\mu$ as a consequence of point 2).
 \end{enumerate}
 The existence of a solution to \eqref{eq:discrete} is an immediate consequence of Lemma \ref{lem:top.degree}.
 Observing that, if $(\uu[h],p_h)\in\Wh$ solves \eqref{eq:discrete}, then $\Psi((\uu[h],p_h),1)=(\underline{\vec{0}},0)$, the bounds~\eqref{eq:a-priori} follow from point 2) above.
\end{proof}

We next consider uniqueness, which can be classically proved under a data smallness condition. 
\begin{theorem}[Uniqueness of the discrete solution]\label{thm:uniqueness}
 Assume that the right-hand side verifies
 \begin{equation}\label{eq:small.data:uniqueness}
   \norm[L^2(\Omega)^d]{\vf}\le\frac{\nu^2}{2C_a^2C_tC_s}
 \end{equation} 
 with $C_a$, $C_t$ and $C_s$ as in \eqref{eq:ah.stab.cont}, \eqref{eq:th.cont}, \eqref{eq:sobolev} respectively.
 Then, the solution $(\uu[h],p_h)\in\UhD\times\Ph$ of \eqref{eq:discrete} is unique.
\end{theorem}

\begin{proof}
 Let $(\uu[1,h],p_{1,h})\in\UhD\times\Ph$ and $(\uu[2,h],p_{2,h})\in\UhD\times\Ph$ solve \eqref{eq:discrete}, and let 
 \[
 \text{$\uw[h]\eqbydef\uu[1,h]-\uu[2,h]$ and $r_h\eqbydef p_{1,h}-p_{2,h} $.}
 \]
 Taking the difference of the discrete momentum balance equation \eqref{eq:discrete:momentum} written for $(\uu[h],p_h)=(\uu[1,h],p_{1,h})$ and $(\uu[h],p_h)=(\uu[2,h],p_{2,h})$, we infer that it holds for all $\uv[h]\in\UhD$,
 \begin{equation}\label{eq:uniqueness:1}
  \nu a_h(\uw[h],\uv[h]) + t_h(\uu[1,h],\uw[h],\uv[h])+ t_h(\uw[h],\uu[2,h],\uv[h]) + b_h(\uv[h],r_h) =0.
 \end{equation}
 Making $\uv[h]=\uw[h]$ in the above equation, observing that $t_h(\uu[1,h],\uw[h],\uw[h])=0$ owing to skew-symmetry (cf. point 1) in Proposition~\ref{prop:th.properties}), that $b_h(\uw[h],r_h)=0$ (this is a consequence of the discrete mass balance equation \eqref{eq:discrete:mass} written for $\uu[1,h]$ and $\uu[2,h]$ with $q_h=r_h$), and using the boundedness \eqref{eq:ah.stab.cont} of $a_h$ and \eqref{eq:th.cont} of $t_h$, we obtain 
 \[
 \left(\nu C_a^{-1}-C_t  \norm[1,h]{\uu[2,h]}\right)\norm[1,h]{\uw[h]}^2\le0.
 \]
By the a priori bounds \eqref{eq:a-priori} and the assumption \eqref{eq:small.data:uniqueness} on $\vf$, the first factor in the left-hand side is $> 0$.
 As a result, we infer $\uw[h]=\underline{\vec{0}}$, thus proving uniqueness for the velocity.
 Plugging this result into \eqref{eq:uniqueness:1}, it is inferred that for all $\uv[h]\in\UhD$ it holds $b_h(\uv[h],r_h) =0$.
 The inf-sup stability (cf. point \ref{item:bh.inf-sup}) in Proposition \ref{prop:bh.properties}) then gives 
 \[
  \norm[L^2(\Omega)]{r_h }\lesssim\sup_{\uv[h]\in\UhD,\norm[1,h]{\uv[h]}=1} b_h(\uv[h],r_h)=0,
 \]
 which proves uniqueness for the pressure.
\end{proof}

\subsection{Convergence to minimal regularity solutions}

\begin{theorem}[Convergence to minimal regularity solutions]\label{thm:conv.min.reg}
  Let $(\Th)_{h\in{\cal H}}$ denote and admissible mesh sequence as in Section \ref{sec:mesh}, and let $((\uu[h],p_h))_{h\in{\cal H}}$ be such that, for all $h\in{\cal H}$, $(\uu[h],p_h)\in\UhD\times\Ph$ solves \eqref{eq:discrete}.
  Then, it holds up to a subsequence with $(\vu,p)\in\U\times P$ solution of the continuous problem \eqref{eq:weak},
  \begin{enumerate}[1)]
  \item $\vu_h\to\vu$ strongly in $L^p(\Omega)^d$ for all $p\in[1,+\infty)$ if $d=2$, $p\in\lbrack 1,6)$ if $d=3$;\label{item:convu}
  \item $\Gh\uu[h]\to\GRAD\vu$ strongly in $L^2(\Omega)^{d\times d}$;\label{item:convgradu}
  \item $s_h(\uu[h],\uu[h])\to 0$;\label{item:convsautu}
   \item $p_h\to p$ strongly in $L^2(\Omega)$.\label{item:convp}
  \end{enumerate}
  If, in addition, the solution to~\eqref{eq:weak} is unique (which is the case if the smallness condition detailed in \cite[Eq. (2.12), Chapter IV]{Girault.Raviart:86} holds for $\vf$), convergence extends to the whole sequence.
\end{theorem}
\begin{proof}
  The proof proceeds in several steps. In {\bf Step 1} we prove the existence of a limit for the sequence of discrete solutions. In {\bf Step 2} we show that this limit is indeed a solution of the continuous problem \eqref{eq:weak}. In {\bf Step 3} we prove the strong convergence of the velocity gradient and of the jumps, and in {\bf Step 4} the strong convergence of the pressure.
  
  \begin{asparaenum}[\bf Step 1.]
  \item \emph{Existence of a limit.}
    Since, for all $h\in{\cal H}$, $(\uu[h],p_h)\in\UhD\times\Ph$ solves \eqref{eq:discrete}, we infer combining the a priori bound \eqref{eq:a-priori} and point 3) in Proposition \ref{prop:Gh.properties} that there exists $(\vu,p)\in U\times P$ such that
    \begin{enumerate}[(i)]
    \item $\vu_h\to\vu$ strongly in $L^p(\Omega)^d$ for all $p\in[1,+\infty)$ if $d=2$, $p\in\lbrack 1,6)$ if $d=3$;
    \item $\Gh[l]\uu[h]\rightharpoonup\GRAD\vu$ weakly in $L^2(\Omega)^{d\times d}$ for all $l\ge 0$;
     \item $p_h\rightharpoonup p$ weakly in $L^2(\Omega)$.
    \end{enumerate}
    \medskip
    
  \item \emph{Identification of the limit.} We next prove that $(\vu,p)\in\U\times P$ is a solution of \eqref{eq:weak}.
    Let $\vPhi\in C_{\rm c}^\infty(\Omega)^d$. We apply the sequential consistency of the viscous, convective and pressure terms 
    (respectively expressed by point \ref{item:ah.seqconsistency}) in Proposition \ref{prop:ah.properties}, point 4) in Proposition \ref{prop:th.properties} and point \ref{item:bh.seq.cons}) in Proposition \ref{prop:bh.properties}) to infer
    \[
    \nu a_h(\uu[h],\Ih\vPhi)+t_h(\uu[h],\uu[h],\Ih\vPhi)+b_h(\Ih\vPhi,p_h)
    \to\nu a(\vu,\vPhi)+t(\vu,\vu,\vPhi)+ b(\vPhi,p).
    \]
    Furthemore, we have $\vproj{k}\vPhi\to\vPhi$ strongly in $L^2(\Omega)^d$, which implies
    \[ 
    \int_\Omega\vf\SCAL \vproj[h]{k}\vPhi \to\int_\Omega\vf\SCAL \vPhi.
    \]
    Finally, point \ref{item:bh.seq.cons}) of Proposition \ref{prop:bh.properties} gives for all $\varphi\in C_c^\infty(\Omega)$
    \[
    b_h(\uu[h],\lproj{k}\varphi)\to b(\vu,\varphi).
    \]
    As a result, we can conclude by density that $(\vu,p)\in\U\times P$ is a solution of \eqref{eq:weak} and point \ref{item:convu}) is proved.
    \medskip
    
  \item \emph{Strong convergence of the velocity gradient and of the jumps.} 
    Making $\uv[h]=\uu[h]$ in \eqref{eq:discrete:momentum} and observing that $t_h(\uu[h],\uu[h],\uu[h])=0$ owing to skew-symmetry \eqref{eq:th.ssymm}
    and that $b_h(\uu[h],p_h)=0$ owing to \eqref{eq:discrete:mass} with $q_h=p_h$, we have
    $$
    \nu \norm[L^2(\Omega)^{d\times d}]{\Gh[k]\uu[h]}^2\le \nu a_h(\uu[h],\uu[h] ) = \int_\Omega\vf\SCAL\vu_h.
    $$
    Since $\vu_h $ converges to $\vu$ strongly in $L^2(\Omega)^d$ and $\vu$ is a solution of \eqref{eq:weak}, we have 
    \[
    \nu \limsup\norm[L^2(\Omega)^{d\times d}]{\Gh[k]\uu[h]}^2
    \le \limsup\int_\Omega\vf\SCAL\vu_h=\int_\Omega\vf\SCAL\vu=\nu \norm[L^2(\Omega)^{d\times d}]{\GRAD\vu}^2.
    \]
    This estimate combined with the weak convergence of $\Gh[k]\uu[h]$ to $\GRAD\vu$ implies the strong convergence of the velocity gradient $\Gh[k]\uu[h]$ to $\GRAD\vu$ in $L^2(\Omega)^{d\times d}$.
    On the other hand, we also obtain that $a_h(\uu[h],\uu[h] )$ converges to $\norm[L^2(\Omega)^{d\times d}]{\GRAD\vu}^2 $, and finally we get
    \begin{equation}\label{eq:convsh}
      s_h(\uu[h],\uu[h])=a_h(\uu[h],\uu[h])-\int_\Omega\Gh[k]\uu[h]\SSCAL\Gh[k]\uu[h] \to 0.
    \end{equation}
    This proves points \ref{item:convgradu}) and \ref{item:convsautu}).
    \medskip
    
  \item \emph{Strong convergence of the pressure.}
    Observing that $p_h\in P$, from the surjectivity of the continuous divergence operator from $\U$ to $P$ we infer the existence of $\vv_{p_h}\in\U$ such that 
    \begin{equation}\label{eq:vp}
      \text{$\DIV\vv_{p_h}=p_h$ and $\norm[H^1(\Omega)^d]{\vv_{p_h}}\lesssim\norm[L^2(\Omega)]{p_h}$.}
    \end{equation}
    We let, for all $h\in{\cal H}$, $\uhv[p_h,h]\eqbydef\Ih\vv_{p_h}$, and study the properties of the sequence $(\uhv[p_h,h])_{h\in{\cal H}}$.
    For all $h\in{\cal H}$, it holds
    \begin{equation}\label{eq:bnd.uhv.ph}
      \norm[1,h]{\uhv[p_h,h]}
      \lesssim\norm[H^1(\Omega)^d]{\vv_{p_h}}
      \lesssim\norm[L^2(\Omega)]{p_h}
      \lesssim\norm[L^2(\Omega)^d]{\vf} + \nu^{-2}\norm[L^2(\Omega)^d]{\vf}^2,
    \end{equation}
    where we have used the boundedness~\eqref{eq:Ih.cont} of $\Ih$ in the first inequality,~\eqref{eq:vp} in the second, and the a priori bound~\eqref{eq:a-priori} on the pressure to conclude.
    Then, by point 3) in Proposition~\ref{prop:Gh.properties}, there exists $\vv_p\in\U$ such that $\vhv_{p_h,h}\to\vv_p$ strongly in $L^p(\Omega)^d$ for all $p\in[1,4]$ and $\Gh[l]\uhv[p_h,h]\rightharpoonup\GRAD\vv_p$ weakly in $L^2(\Omega)^{d\times d}$ for all $l\ge 0$.
    Moreover, by uniqueness of the limit in the distribution sense, it holds that
    \begin{equation}\label{eq:div.vv_p=p}
      \DIV\vv_p = p
    \end{equation}
    Making $\uv[h]=\uhv[p_h,h]$ in the discrete momentum balance equation \eqref{eq:discrete:momentum} and recalling the commuting property \eqref{eq:Dh.commuting}, we have
    \begin{equation}\label{eq:egnormp}
      \norm[L^2(\Omega)]{p_h}^2
      = -b_h(\uhv[p_h,h] ,p_h)
      =   \nu a_h(\uu[h],\uhv[p_h,h]) + t_h(\uu[h],\uu[h],\uhv[p_h,h])-\int_\Omega\vf\SCAL\vhv_{p_h,h}.
    \end{equation}   
    We study the limit of the three terms on the right of~\eqref{eq:egnormp} using the convergence properties for the discrete solution proved in the previous points.
    Combining the strong converge of $\Gh[k]\uu[h]$ with the weak convergence of $\Gh[k]\uhv[p_h,h]$ gives
    \[
    \int_\Omega\Gh[k]\uu[h]\SSCAL\Gh[k]\uhv[p_h,h] \to  \int_\Omega\GRAD\vu\SSCAL\GRAD\vv_p.
    \]
    Moreover, the convergence~\eqref{eq:convsh} of the jumps of $\uu[h]$ and the uniform bound \eqref{eq:bnd.uhv.ph} imply
    \[
    s_h(\uu[h],\uhv[p_h,h])\le  s_h(\uu[h],\uu[h])^{\frac12}\times s_h(\uhv[p_h,h],\uhv[p_h,h])^{\frac12}\to 0,
    \]
    so that, in conclusion, we have for the viscous term
    \[
    a_h(\uu[h],\uhv[p_h,h])\to  a(\vu,\vv_p).
    \]
    Observing that the convergence properties of the sequences $(\uu[h])_{h\in{\cal H}}$ and $(\uhv[p_h,h])_{h\in{\cal H}}$ are respectively analogous to those of the sequences $(\Ih\vPhi)_{h\in{\cal H}}$ and $(\uv[h])_{h\in{\cal H}}$ in point 4) of Proposition~\ref{prop:th.properties}, we can prove proceeding in a similar way that
    \[
    t_h(\uu[h],\uu[h],\uhv[p_h,h]) \to t(\vu,\vu, \vv_p).
    \]    
    Finally, by strong convergence of $\vhv_{p_h,h}$ to $\vv_{p}$ in $L^2(\Omega)^d$, we readily infer for the source term
    \[
    \int_\Omega\vf\SCAL\vhv_{p_h,h} \to \int_\Omega\vf\SCAL \vv_p.     
    \]
    Collecting the above convergence results and using the momentum balance equation \eqref{eq:weak:momentum} together with~\eqref{eq:div.vv_p=p} leads to
    \[
    \limsup  \norm[L^2(\Omega)]{p_h}^2 \le \nu a(\vu,\vv_p)+t(\vu,\vu, \vv_p)-\int_\Omega\vf\SCAL \vv_p=-b(\vv_p,p)=\norm[L^2(\Omega)]{p}^2,
    \]
    and the strong convergence of the pressure in $L^2(\Omega)$ stated in point \ref{item:convp}) follows.\qedhere
  \end{asparaenum}
\end{proof}

\subsection{Convergence rates for small data}

\begin{theorem}[Convergence rates for small data]\label{thm:err.est}
  Let $(\vu,p)\in\U\times P$ and $(\uu[h],p_h)\in\UhD\times\Ph$ solve problems \eqref{eq:weak} and \eqref{eq:discrete}, respectively, and assume uniqueness (which holds, in particular, if both smallness conditions \cite[Eq. (2.12), Chapter IV]{Girault.Raviart:86} and~\eqref{eq:small.data:uniqueness} are verified).
  Assume, moreover, the additional regularity $(\vu,p)\in H^{k+2}(\Omega)^d\times H^{k+1}(\Omega)$, as well as
  \begin{equation}\label{eq:small.data}
    \norm[L^2(\Omega)^d]{\vf}\le\frac{\nu^2}{2C_IC_aC_t(1+C_{\rm P}^2)},
  \end{equation}
  with $C_I$, $C_a$ and $C_t$ as in \eqref{eq:Ih.cont}, \eqref{eq:ah.stab.cont} and \eqref{eq:th.cont}, respectively, and $C_{\rm P}$ Poincar\'{e} constant only depending on $\Omega$ such that, for all $\vv\in\U$, $\norm[L^2(\Omega)^d]{\vv}\le C_{\rm P}\norm[L^2(\Omega)^{d\times d}]{\GRAD\vv}$.
  Let
  $$
  \uhu[h]\eqbydef\Ih\vu,\qquad \hph\eqbydef\lproj{k} p.
  $$  
  Then, there is $C>0$ independent of both $h$ and $\nu$ such that
  \begin{equation}\label{eq:err.est}
    \norm[1,h]{\uu[h]-\uhu[h]}
    + \nu^{-1}\norm[L^2(\Omega)]{p_h-\hph}
    \le C h^{k+1} \big(
    \left(1+\nu^{-1}\norm[H^2(\Omega)^d]{\vu}\right)\norm[H^{k+2}(\Omega)^d]{\vu} + \nu^{-1}\norm[H^{k+1}(\Omega)]{p}
    \big).
  \end{equation}
\end{theorem}

\begin{corollary}[Convergence rates for small data]
  Under the above assumptions, it holds
  \begin{multline*}
    \norm[L^2(\Omega)^{d\times d}]{\Gh\uu[h]-\GRAD\vu} + 
    s_h(\uu[h],\uu[h])^{\nicefrac12} + \nu^{-1}\norm[L^2(\Omega)]{p_h-p}
    \lesssim
    \\
    h^{k+1} \big(
    \left(1+\nu^{-1}\norm[H^2(\Omega)^d]{\vu}\right)\norm[H^{k+2}(\Omega)^d]{\vu} + \nu^{-1}\norm[H^{k+1}(\Omega)]{p}
    \big),
  \end{multline*}
  where the second term in the left-hand side accounts for the jumps of the discrete solution.
\end{corollary}
\begin{proof}
  Using the triangle inequality, we infer
  \begin{multline*}
  \norm[L^2(\Omega)^{d\times d}]{\Gh\uu[h]-\GRAD\vu} + s_h(\uu[h],\uu[h])^{\nicefrac12} + \nu^{-1}\norm[L^2(\Omega)]{p_h-p}
  \le
  \\
  \norm[L^2(\Omega)^{d\times d}]{\GRAD\vu-\Gh\uhu[h]} + s_h(\uhu[h],\uhu[h])^{\nicefrac12} + \nu^{-1}\norm[L^2(\Omega)]{p-\hph}
  \\
  + \norm[a,h]{\uu[h]-\uhu[h]} + \nu^{-1}\norm[L^2(\Omega)]{p_h-\hph}.
  \end{multline*}
  The terms in the second line can be estimated recalling the consistency property~\eqref{eq:Gh.approx} for the gradient reconstruction and using the approximation properties \eqref{eq:lproj.approx} of the $L^2$-orthogonal projector and the consistency properties~\eqref{eq:sh.consistency} of $s_h$.
  For the terms in the third line, recall the norm equivalence~\eqref{eq:ah.stab.cont} and use~\eqref{eq:err.est}.
\end{proof}

\begin{remark}[Extension to other hybrid discretizations]
  The following proof  extends without modifications to any bilinear forms $a_h$ and $b_h$ and trilinear form $t_h$ that match, respectively, the properties 1) and 2) in Proposition~\ref{prop:ah.properties}, 1)--3) in Proposition~\ref{prop:bh.properties}, and 1) and 2) in Proposition~\ref{prop:th.properties}, respectively.
  Such properties can therefore be intended as design guidelines.
\end{remark}

\begin{proof}[Proof of Theorem~\ref{thm:err.est}]
  Let, for the sake of brevity, $(\ue,\eps)\eqbydef(\uu[h]-\uhu[h],p_h-\hph)$.
  The proof proceeds in three steps: in {\bf Step 1}, we identify the consistency error and derive a lower bound in terms of $\norm[1,h]{\ue[h]}$ using the data smallness assumption,
  in {\bf Step 2} we estimate the error on the velocity and in {\bf Step 3} the error on the pressure.

  \begin{asparaenum}[\bf Step 1.]
  \item \emph{Consistency error and lower bound.}
    It is readily inferred from the discrete momentum balance equation \eqref{eq:discrete:momentum} that it holds, for all $\uv[h]\in\UhD$,
    \begin{equation}\label{eq:err.eq:momentum}
      \nu a_h(\ue,\uv[h]) + t_h(\uu[h],\uu[h],\uv[h]) - t_h(\uhu[h],\uhu[h],\uv[h]) + b_h(\uv[h],\eps) = \Eh(\uv[h]),
    \end{equation}
    with consistency error
    $$
    \Eh(\uv[h])\eqbydef (\vf,\vv_h) - a_h(\uhu[h],\uv[h]) - t_h(\uhu[h],\uhu[h],\uv[h]) - b_h(\uv[h],\hph).
    $$
    Making $\uv[h]=\ue$ in \eqref{eq:err.eq:momentum}, and observing that $t_h(\uu[h],\uu[h],\ue)=t_h(\uu[h],\uhu[h],\ue)$ owing to the skew-symmetry property \eqref{eq:th.ssymm}, and that $b_h(\ue,\eps)=b_h(\uu[h],\eps)-b_h(\uhu[h],\eps)=0$ owing to \eqref{eq:discrete:mass} and since $\Dh\uhu[h]=\lproj{k}(\DIV\vu)=0$ (cf. \eqref{eq:Dh.commuting} and \eqref{eq:strong:mass}), we infer
    \begin{equation}\label{eq:lower.bnd.Eh}
      \begin{aligned}
        \Eh(\ue) &=\nu\norm[a,h]{\ue}^2 + t_h(\ue,\uhu[h],\ue)
        \\
        &\ge\nu C_a^{-1}\norm[1,h]{\ue}^2 - C_t\norm[1,h]{\uhu[h]}\norm[1,h]{\ue}^2
        \\
        &\ge\left(\nu C_a^{-1} - C_tC_I(1+C_{\rm P}^2)\nu^{-1}\norm[H^{-1}(\Omega)^d]{\vf}\right)\norm[1,h]{\ue}^2
        \gtrsim\nu\norm[1,h]{\ue}^2,
      \end{aligned}
    \end{equation}
    where we have used the coercivity of $a_h$ expressed by the first inequality in \eqref{eq:ah.stab.cont} together with the boundedness \eqref{eq:th.cont} of $t_h$ to pass to the second line, the boundedness \eqref{eq:Ih.cont} of $\Ih$ with the standard a priori estimate $\norm[H^1(\Omega)^d]{\vu}\le(1+C_{\rm P}^2)\nu^{-1}\norm[H^{-1}(\Omega)^d]{\vf}$ on the exact solution to infer
    \begin{equation}\label{eq:a-priori:uhu}
      \norm[1,h]{\uhu[h]}\le C_I\norm[H^1(\Omega)^d]{\vu}\le C_I(1+C_{\rm P}^2)\nu^{-1}\norm[H^{-1}(\Omega)^d]{\vf},
    \end{equation}
    and the data smallness assumption \eqref{eq:small.data} to conclude.

  \item \emph{Estimate on the velocity.}
    Observing that $\vf=-\nu\LAPL\vu + \GRAD\vu~\vu + \GRAD p$ a.e. in $\Omega$ (cf. \eqref{eq:strong:momentum}), it holds for all $\uv[h]\in\UhD$,
    $$
    |\Eh(\uv[h])| \le 
    \nu\left| \int_\Omega(\LAPL\vu)\SCAL\vv_h + a_h(\uhu[h],\uv[h])\right|
    + \left|\int_\Omega\vv_h\trans\GRAD\vu\MVPROD\vu - t_h(\uhu[h],\uhu[h],\uv[h])\right|
    + \left|\int_\Omega\GRAD p\SCAL\vv_h - b_h(\uv[h],\hph)\right|.
    $$
    Using \eqref{eq:ah.consistency}, \eqref{eq:th.consistency} and \eqref{eq:bh.consistency}, respectively, to estimate the three terms in the right-hand side, it is readily inferred that
    \begin{equation}\label{eq:sup.Eh}
      \mathrm{S}\eqbydef\sup_{\uv[h]\in\UhD,\norm[1,h]{\uv[h]}=1}|\Eh(\uv[h])|
      \lesssim
      \nu h^{k+1}\left(1 + \nu^{-1}\norm[H^2(\Omega)^d]{\vu}\right)\norm[H^{k+2}(\Omega)^d]{\vu} + h^{k+1}\norm[H^{k+1}(\Omega)]{p},
    \end{equation}
    so that, in particular,
    \begin{equation}\label{eq:upper.bnd.Eh}
      |\Eh(\ue)|\le\mathrm{S}\norm[1,h]{\ue}
      \lesssim\left[
      \nu h^{k+1}\left(1 + \nu^{-1}\norm[H^2(\Omega)^d]{\vu}\right)\norm[H^{k+2}(\Omega)^d]{\vu} + h^{k+1}\norm[H^{k+1}(\Omega)]{p}
      \right]\norm[1,h]{\ue}.
    \end{equation}
    Combining \eqref{eq:lower.bnd.Eh} with \eqref{eq:upper.bnd.Eh}, the estimate on the velocity in \eqref{eq:err.est} follows.

  \item \emph{Estimate on the pressure.}
    Let us now estimate the error on the pressure. We have
    \begin{equation}\label{eq:err.est:basic.p}
      \begin{aligned}
        \norm[L^2(\Omega)]{\eps}
        &\lesssim\sup_{\uv[h]\in\UhD,\norm[1,h]{\uv[h]}=1} b_h(\uv[h],\eps)
        \\
        &=\sup_{\uv[h]\in\UhD,\norm[1,h]{\uv[h]}=1}\left[
        \Eh(\uv[h]) - \nu a_h(\ue[h],\uv[h]) - t_h(\uu[h],\uu[h],\uv[h]) +  t_h(\uhu[h],\uhu[h],\uv[h]) 
        \right]
        \\
        &=\sup_{\uv[h]\in\UhD,\norm[1,h]{\uv[h]}=1}\left[
        \Eh(\uv[h]) - \nu a_h(\ue[h],\uv[h]) - t_h(\ue[h],\uu[h],\uv[h]) -  t_h(\uhu[h],\ue[h],\uv[h]) 
        \right]
        \\
        &\lesssim\mathrm{S} + \nu\left(1 + \nu^{-1}\norm[1,h]{\uu[h]} + \nu^{-1}\norm[1,h]{\uhu[h]}\right)\norm[1,h]{\ue[h]}
        \\
        &\lesssim\mathrm{S} + \nu\left(1 + \nu^{-2}\norm[L^2(\Omega)^d]{\vf} + \nu^{-2}\norm[H^{-1}(\Omega)^d]{\vf}\right)\norm[1,h]{\ue[h]}
        \lesssim\mathrm{S} + \nu\norm[1,h]{\ue[h]},
      \end{aligned}
    \end{equation}
    In \eqref{eq:err.est:basic.p}, we have used the inf-sup inequality \eqref{eq:bh.inf-sup} on $b_h$ in the first line and the error equation \eqref{eq:err.eq:momentum} to pass to the second line;
    to pass to the third line, we have inserted $\pm t_h(\uhu[h],\uu[h],\uv[h])$ and used the linearity of $t_h$ in its first and second arguments;
    to pass to the fourth line, we have used the boundedness \eqref{eq:ah.stab.cont} of $a_h$ and \eqref{eq:th.cont} of $t_h$;
    to pass to the fifth line, we have used the a priori bounds \eqref{eq:a-priori} on $\norm[1,h]{\uu[h]}$ and \eqref{eq:a-priori:uhu} on $\norm[1,h]{\uhu[h]}$;
    the data smallness assumption \eqref{eq:small.data} gives the conclusion.
    The estimate on the pressure then follows using \eqref{eq:sup.Eh} and \eqref{eq:err.est}, respectively, to further bound the addends in the right-hand side of \eqref{eq:err.est:basic.p}.\qedhere
  \end{asparaenum}
\end{proof}

\subsection{Numerical example}

To close this section, we provide a numerical example that demonstrates the convergence properties of our method.
We solve on the two-dimensional square domain $\Omega\eqbydef(-0.5,1.5)\times(0,2)$ the Dirichlet problem corresponding to the exact solution $(\vu,p)$ of \cite{Kovasznay:48} with $\vu=(u_1,u_2)$ such that, introducing the Reynolds number $\Reynolds\eqbydef(2\nu)^{-1}$ and letting $\lambda\eqbydef\Reynolds-\left(\Reynolds^2 + 4\pi^2\right)^{\nicefrac12}$,
$$
u_1(\vec{x}) \eqbydef 1-\exp(\lambda x_1)\cos(2\pi x_2),\qquad
u_2(\vec{x}) \eqbydef \frac{\lambda}{2\pi}\exp(\lambda x_1)\sin(2\pi x_2),
$$
and pressure given by
$$
p(\vec{x}) \eqbydef -\frac12\exp(2\lambda x_1) + \frac{\lambda}{2}\left(\exp(4\lambda)-1\right).
$$
We take here $\nu=1$ and consider two sequences of refined meshes obtained by linearly mapping on $\Omega$ the mesh family 2 of~\cite{Herbin.Hubert:08} and the (predominantly) hexagonal mesh family of \cite{Di-Pietro.Lemaire:15} (both meshes were originally defined on the unit square).
The implementation uses the static condensation procedure discussed in Remark~\ref{rem:static.cond}.
The convergence results for $k=2$ and $k=3$ are reported in Figures~\ref{fig:cartesian} and~\ref{fig:hexagonal}, respectively.
Using the notation of Theorem~\ref{thm:err.est}, we separately plot the $H^1$-error on the velocity $\norm[1,h]{\uu[h]-\uhu[h]}$, the $L^2$-error on the pressure $\norm[L^2(\Omega)]{p_h-\hph}$, and the $L^2$-error on the velocity $\norm[L^2(\Omega)^d]{\vu_h-\vhu_h}$.
For the sake of completeness, we consider both the trilinear forms $t_h$ given by \eqref{eq:th} and $t_h^{\rm HDG}$ given by \eqref{eq:th.HDG} (with $\eta=0$).
In both cases, we obtain similar results, and the $H^1$-error on the velocity as well as the $L^2$-error on the pressure converge as $h^{k+1}$ as expected.
The $L^2$-error on the velocity, on the other hand, converges as $h^{k+2}$.
Notice also that this means that the error $\norm[L^2(\Omega)^d]{\ph\uu[h]-\vu}$ can be proved to converge as $h^{k+2}$ following a similar reasoning as in~\cite[Corollary~4.6]{Aghili.Boyaval.ea:15}.
The details are omitted for the sake of brevity.

%------------------------------------------------------------------------------%
% Cartesian mesh family

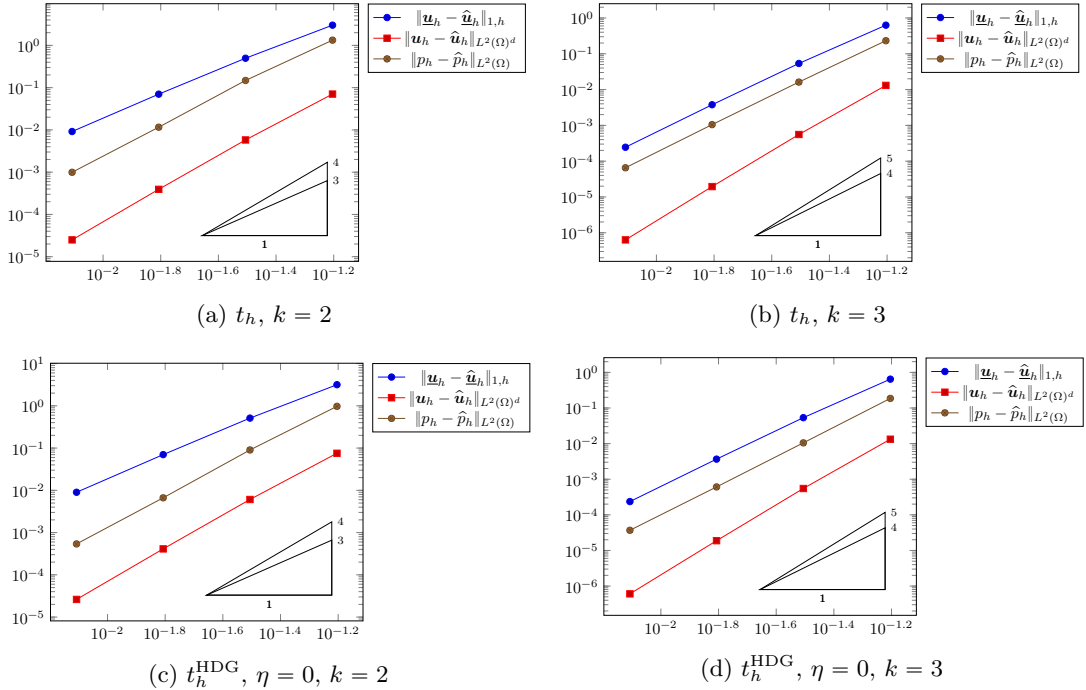
\begin{figure}[ht]\centering
  %% DPK k=2  
  \begin{minipage}{0.45\textwidth}\centering
    \begin{tikzpicture}[scale=0.60]
      \begin{loglogaxis}[ legend pos = outer north east ]
        \addplot table[x=meshsize,y=err_u]{../dat/dpk/ns_2_mesh2.dat};
        \addplot table[x=meshsize,y=err_l2_u]{../dat/dpk/ns_2_mesh2.dat};    
        \addplot table[x=meshsize,y=err_p]{../dat/dpk/ns_2_mesh2.dat};

        \logLogSlopeTriangle{0.90}{0.4}{0.1}{3}{black};
        \logLogSlopeTriangle{0.90}{0.4}{0.1}{4}{black};
        \legend{{$\norm[1,h]{\uu[h]-\uhu[h]}$},{$\norm[L^2(\Omega)^d]{\vu_h-\vhu_h}$},{$\norm[L^2(\Omega)]{p_h-\hph}$}}
      \end{loglogaxis} 
    \end{tikzpicture}
    \subcaption{$t_h$, $k=2$}    
  \end{minipage}
  %% DPK k=3
  \begin{minipage}{0.45\textwidth}\centering
    \begin{tikzpicture}[scale=0.60]
      \begin{loglogaxis}[ legend pos = outer north east ]
        \addplot table[x=meshsize,y=err_u]{../dat/dpk/ns_3_mesh2.dat};
        \addplot table[x=meshsize,y=err_l2_u]{../dat/dpk/ns_3_mesh2.dat};    
        \addplot table[x=meshsize,y=err_p]{../dat/dpk/ns_3_mesh2.dat};

        \logLogSlopeTriangle{0.90}{0.4}{0.1}{4}{black};
        \logLogSlopeTriangle{0.90}{0.4}{0.1}{5}{black};
        \legend{{$\norm[1,h]{\uu[h]-\uhu[h]}$},{$\norm[L^2(\Omega)^d]{\vu_h-\vhu_h}$},{$\norm[L^2(\Omega)]{p_h-\hph}$}}
      \end{loglogaxis}
    \end{tikzpicture}
    \subcaption{$t_h$, $k=3$}    
  \end{minipage}
  \vspace{0.25cm} \\
  %% QS k=2
  \begin{minipage}{0.45\textwidth}\centering
    \begin{tikzpicture}[scale=0.60]
      \begin{loglogaxis}[ legend pos = outer north east ]
        \addplot table[x=meshsize,y=err_u]{../dat/qs/ns_2_mesh2.dat};
        \addplot table[x=meshsize,y=err_l2_u]{../dat/qs/ns_2_mesh2.dat};    
        \addplot table[x=meshsize,y=err_p]{../dat/qs/ns_2_mesh2.dat};

        \logLogSlopeTriangle{0.90}{0.4}{0.1}{3}{black};
        \logLogSlopeTriangle{0.90}{0.4}{0.1}{4}{black};
        \legend{{$\norm[1,h]{\uu[h]-\uhu[h]}$},{$\norm[L^2(\Omega)^d]{\vu_h-\vhu_h}$},{$\norm[L^2(\Omega)]{p_h-\hph}$}}        
      \end{loglogaxis}
    \end{tikzpicture}
    \subcaption{$t_h^{\rm HDG}$, $\eta=0$, $k=2$}    
  \end{minipage}
  %% QS k=3
  \begin{minipage}{0.45\textwidth}\centering
    \begin{tikzpicture}[scale=0.60]
      \begin{loglogaxis}[ legend pos = outer north east ]
        \addplot table[x=meshsize,y=err_u]{../dat/qs/ns_3_mesh2.dat};
        \addplot table[x=meshsize,y=err_l2_u]{../dat/qs/ns_3_mesh2.dat};    
        \addplot table[x=meshsize,y=err_p]{../dat/qs/ns_3_mesh2.dat};

        \logLogSlopeTriangle{0.90}{0.4}{0.1}{4}{black};
        \logLogSlopeTriangle{0.90}{0.4}{0.1}{5}{black};
        \legend{{$\norm[1,h]{\uu[h]-\uhu[h]}$},{$\norm[L^2(\Omega)^d]{\vu_h-\vhu_h}$},{$\norm[L^2(\Omega)]{p_h-\hph}$}}        
      \end{loglogaxis}
    \end{tikzpicture}
    \subcaption{$t_h^{\rm HDG}$, $\eta=0$, $k=3$}    
  \end{minipage}
  \caption{Cartesian mesh family, errors versus $h$. The triangles indicate reference slopes. The trilinear forms $t_h$ and $t_h^{\rm HDG}$ are defined by~\eqref{eq:th} and~\eqref{eq:th.HDG}, respectively.\label{fig:cartesian}}
\end{figure}

%------------------------------------------------------------------------------%
% Hexagonal mesh family

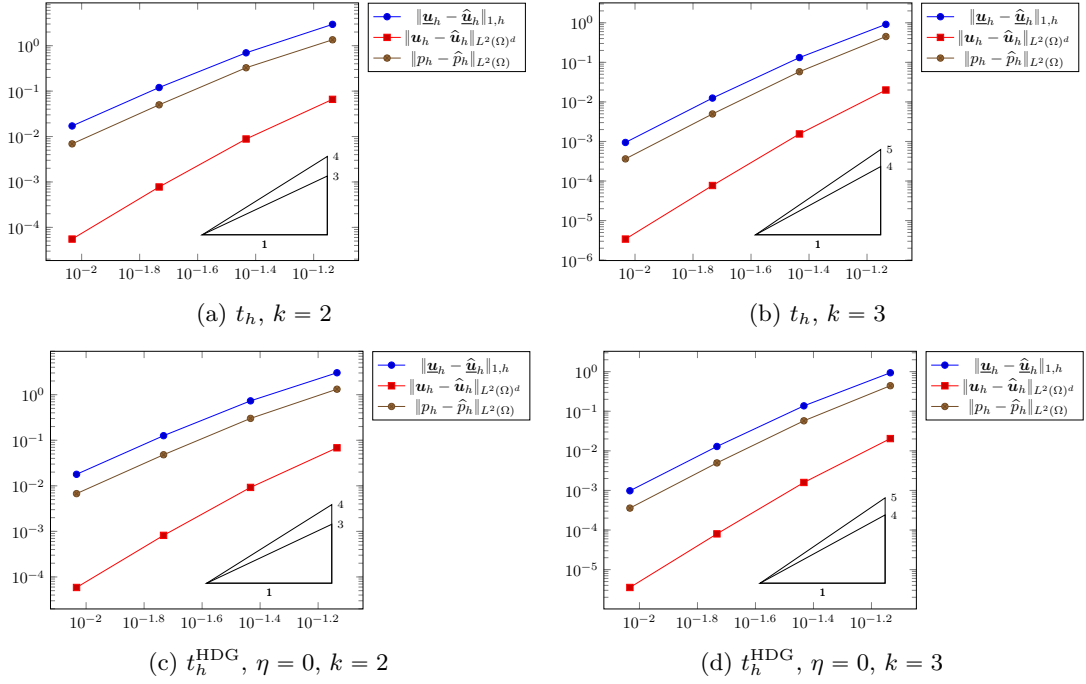
\begin{figure}[ht]\centering
  \centering
  %% DPK k=2
  \begin{minipage}{0.45\textwidth}\centering
    \begin{tikzpicture}[scale=0.60]
      \begin{loglogaxis}[ legend pos = outer north east ]
        \addplot table[x=meshsize,y=err_u]{../dat/dpk/ns_2_pi6_tiltedhexagonal.dat};
        \addplot table[x=meshsize,y=err_l2_u]{../dat/dpk/ns_2_pi6_tiltedhexagonal.dat};    
        \addplot table[x=meshsize,y=err_p]{../dat/dpk/ns_2_pi6_tiltedhexagonal.dat};

        \logLogSlopeTriangle{0.90}{0.4}{0.1}{3}{black};
        \logLogSlopeTriangle{0.90}{0.4}{0.1}{4}{black};
        \legend{{$\norm[1,h]{\uu[h]-\uhu[h]}$},{$\norm[L^2(\Omega)^d]{\vu_h-\vhu_h}$},{$\norm[L^2(\Omega)]{p_h-\hph}$}}        
      \end{loglogaxis}
    \end{tikzpicture}
    \subcaption{$t_h$, $k=2$}
  \end{minipage}
  %% DPK k=3
  \begin{minipage}{0.45\textwidth}\centering
    \begin{tikzpicture}[scale=0.60]
      \begin{loglogaxis}[ legend pos = outer north east ]
        \addplot table[x=meshsize,y=err_u]{../dat/dpk/ns_3_pi6_tiltedhexagonal.dat};
        \addplot table[x=meshsize,y=err_l2_u]{../dat/dpk/ns_3_pi6_tiltedhexagonal.dat};    
        \addplot table[x=meshsize,y=err_p]{../dat/dpk/ns_3_pi6_tiltedhexagonal.dat};

        \logLogSlopeTriangle{0.90}{0.4}{0.1}{4}{black};
        \logLogSlopeTriangle{0.90}{0.4}{0.1}{5}{black};
        \legend{{$\norm[1,h]{\uu[h]-\uhu[h]}$},{$\norm[L^2(\Omega)^d]{\vu_h-\vhu_h}$},{$\norm[L^2(\Omega)]{p_h-\hph}$}}        
      \end{loglogaxis}
    \end{tikzpicture}
    \subcaption{$t_h$, $k=3$}    
  \end{minipage}
  \vspace{0.25cm} \\
  %% QS k=2
  \begin{minipage}{0.45\textwidth}\centering
    \begin{tikzpicture}[scale=0.60]
      \begin{loglogaxis}[ legend pos = outer north east ]
        \addplot table[x=meshsize,y=err_u]{../dat/qs/ns_2_pi6_tiltedhexagonal.dat};
        \addplot table[x=meshsize,y=err_l2_u]{../dat/qs/ns_2_pi6_tiltedhexagonal.dat};    
        \addplot table[x=meshsize,y=err_p]{../dat/qs/ns_2_pi6_tiltedhexagonal.dat};

        \logLogSlopeTriangle{0.90}{0.4}{0.1}{3}{black};
        \logLogSlopeTriangle{0.90}{0.4}{0.1}{4}{black};
        \legend{{$\norm[1,h]{\uu[h]-\uhu[h]}$},{$\norm[L^2(\Omega)^d]{\vu_h-\vhu_h}$},{$\norm[L^2(\Omega)]{p_h-\hph}$}}        
      \end{loglogaxis}
    \end{tikzpicture}
    \subcaption{$t_h^{\rm HDG}$, $\eta=0$, $k=2$}    
  \end{minipage}
  %% QS k=3
  \begin{minipage}{0.45\textwidth}\centering
    \begin{tikzpicture}[scale=0.60]
      \begin{loglogaxis}[ legend pos = outer north east ]
        \addplot table[x=meshsize,y=err_u]{../dat/qs/ns_3_pi6_tiltedhexagonal.dat};
        \addplot table[x=meshsize,y=err_l2_u]{../dat/qs/ns_3_pi6_tiltedhexagonal.dat};    
        \addplot table[x=meshsize,y=err_p]{../dat/qs/ns_3_pi6_tiltedhexagonal.dat};

        \logLogSlopeTriangle{0.90}{0.4}{0.1}{4}{black};
        \logLogSlopeTriangle{0.90}{0.4}{0.1}{5}{black};
        \legend{{$\norm[1,h]{\uu[h]-\uhu[h]}$},{$\norm[L^2(\Omega)^d]{\vu_h-\vhu_h}$},{$\norm[L^2(\Omega)]{p_h-\hph}$}}        
      \end{loglogaxis}
    \end{tikzpicture}
    \subcaption{$t_h^{\rm HDG}$, $\eta=0$, $k=3$}
  \end{minipage}
  \caption{Hexagonal mesh family, errors versus $h$. The triangles indicate reference slopes. The trilinear forms $t_h$ and $t_h^{\rm HDG}$ are defined by~\eqref{eq:th} and~\eqref{eq:th.HDG}, respectively.\label{fig:hexagonal}}
\end{figure}

%------------------------------------------------------------------------------%

\section{Properties of the discrete bilinear and trilinear forms}\label{sec:properties}

We gather in this section the proofs of Propositions \ref{prop:ah.properties}, \ref{prop:th.properties} and \ref{prop:bh.properties}.

\subsection{Viscous bilinear form}\label{sec:ah.properties}

\begin{proof}[Proof of Proposition \ref{prop:ah.properties}]
  We only sketch the proof and provide references for the details.

  \begin{asparaenum}[1)]
  \item This norm equivalence follows taking $p=2$ in \cite[Lemma~5.2]{Di-Pietro.Droniou:16}, where the scalar case is considered; cf. also \cite[Lemma 4]{Di-Pietro.Ern.ea:14}, where a slightly different expression for $a_T$ is studied (cf. Remark \ref{rem:ah}).

  \item We adapt the arguments of \cite[Theorem 8]{Di-Pietro.Ern.ea:14}.
    For the sake of brevity, we let $\uhv[h]\eqbydef\Ih\vv$ in what follows.
    Integrating by parts element-by-element, and using the single-valuedness of $\GRAD\vv\MVPROD\normal_F$ at interfaces and the fact that $\vw_F=\vec{0}$ on boundary faces to insert $\vw_F$ into the second term, we have
    \begin{equation}\label{eq:ah.consistency:1}
      \int_\Omega\LAPL\vv\SCAL\vw_h
      = -\sum_{T\in\Th}\left(
      \int_T\GRAD\vv\SSCAL\GRAD\vw_T + \sum_{F\in\Fh[T]} \int_F(\vw_F-\vw_T)\trans\GRAD\vv\MVPROD\normal_{TF}
      \right).
    \end{equation}
    On the other hand, using on each element $T\in\Th$ the definition \eqref{eq:GT'} of $\GT$ (with $\uv[T]=\uw[T]$ and $\mtau=\GT\uhv[T]$), we have
    \begin{equation}\label{eq:ah.consistency:2}
      a_h(\uhv[h],\uw[h])
      = \sum_{T\in\Th}\left(
      \int_T\GT\uhv[T]\SSCAL\GRAD\vw_T + \sum_{F\in\Fh[T]} \int_F(\vw_F-\vw_T)\trans\GT\uhv[T]\MVPROD\normal_{TF}
      \right) + s_{h}(\uhv[h],\uw[h]).
    \end{equation}
    Summing \eqref{eq:ah.consistency:1} and \eqref{eq:ah.consistency:2}, observing that the first terms in parentheses cancel out as a result of the Euler equation \eqref{eq:GT.euler} for $\GT$, and using the Cauchy--Schwarz inequality followed by the trace approximation properties \eqref{eq:Gh.approx} of $\GT$, the consistency properties \eqref{eq:sh.consistency} of $s_h$, and the norm equivalence \eqref{eq:ah.stab.cont}, we get
    $$
    \begin{aligned}
      &\left|
      \int_\Omega\LAPL\vv\SCAL\vw_h + a_h(\uhv[h],\uw[h])
      \right|
      \\
      &\qquad\lesssim
      \left(
      \sum_{T\in\Th} h_T\norm[L^2(\partial T)^{d\times d}]{\GT\uhv[T]-\GRAD\vv}^2
      + s_h(\uhv[h],\uhv[h])
      \right)^{\nicefrac12}\times\left(
      \sum_{T\in\Th}\seminorm[1,\partial T]{\uw[T]}^2
      + s_h(\uw[h],\uw[h])
      \right)^{\nicefrac12}
      \\
      &\qquad\lesssim h^{k+1}\norm[H^{k+2}(\Omega)^d]{\vv}\norm[1,h]{\uw[h]},
    \end{aligned}
    $$
    which concludes the proof of \eqref{eq:ah.consistency}.
    
  \item The sequential consistency can be proved following steps 1) and 2) of \cite[Theorem~4.6]{Di-Pietro.Droniou:16}, where the scalar case is considered.\qedhere
  \end{asparaenum}
\end{proof}

\subsection{Convective trilinear form}\label{sec:th.properties}

\begin{proof}[Proof of Proposition \ref{prop:th.properties}]
  \begin{asparaenum}[1)]
  \item \emph{Skew-symmetry.} This property is straightforward from the definition of $t_h$.
  \item \emph{Boundedness.} For all $\uw[h],\uu[h],\uv[h]\in\UhD$, using H\"{o}lder inequalities, we have
    $$
    \begin{aligned}
      |t_h(\uw[h],\uu[h],\uv[h])|
      &\lesssim\norm[L^4(\Omega)^d]{\vv_h}\norm[L^2(\Omega)^{d\times d}]{\Gh[2k]\uu[h]}\norm[L^4(\Omega)^d]{\vw_h}
      + \norm[L^4(\Omega)^d]{\vu_h}\norm[L^2(\Omega)^{d\times d}]{\Gh[2k]\uv[h]}\norm[L^4(\Omega)^d]{\vw_h}
      \\
      &\lesssim\norm[1,h]{\uv[h]}\norm[1,h]{\uu[h]}\norm[1,h]{\uw[h]},
    \end{aligned}
    $$
    where the conclusion follows using several times the discrete Sobolev embedding \eqref{eq:sobolev} with $p=4$ and the boundedness \eqref{eq:Gh.cont} of $\Gh[2k]$.
    
  \item \emph{Consistency.} Set, for the sake of brevity, $\uhv[h]\eqbydef\Ih\vv$.
    Integrating by parts element-by-element, recalling that $\DIV\vv=0$, and using the single-valuedness of $(\vv\SCAL\normal_F)\vv$ at interfaces together with the fact that $\vw_F=\vec{0}$ on boundary faces to insert $\vw_F$ into the third term, we have
    \begin{equation}\label{eq:th.approx:1}
      \int_\Omega\vw_h\trans\GRAD\vv\MVPROD\vv
      = \frac12\left(
      \int_\Omega\vw_h\trans\GRAD\vv\MVPROD\vv
      - \int_\Omega\vv\trans\GRADh\vw_h\MVPROD\vv
      - \sum_{T\in\Th}\sum_{F\in\Fh[T]}\int_F(\vv\SCAL\normal_{TF})(\vw_F-\vw_T)\SCAL\vv
      \right).
    \end{equation}
    On the other hand, using on each element $T\in\Th$ the definition \eqref{eq:GT'} of $\GT[2k]$ (with $\uv[T]=\uw[T]$ and $\mtau=\vhv_T\otimes\vhv_T$), we have
    \begin{equation}\label{eq:th.approx:2}
      t_h(\uhv[h],\uhv[h],\uw[h])
      = \frac12\left(
      \int_{\Omega}\vw_h\trans\Gh[2k]\uhv[h]\MVPROD\vhv_h
      - \int_{\Omega}\vhv_h\trans\GRADh\vw_h\MVPROD\vhv_h
      - \hspace{-0.75ex}\sum_{T\in\Th}\sum_{F\in\Fh[T]}\int_F(\vhv_T\SCAL\normal_{TF})(\vw_F-\vw_T)\SCAL\vhv_T
      \right).
    \end{equation}
    Subtracting \eqref{eq:th.approx:2} from \eqref{eq:th.approx:1} and inserting in the right-hand side of the resulting expression the quantity
    $$
    \pm\frac12\left(
    \int_\Omega\vw_h\trans\Gh[2k]\uhv[h]\MVPROD\vv
    + \int_\Omega\vv\trans\GRADh\vw_h\MVPROD\vhv_h
    + \sum_{T\in\Th}\sum_{F\in\Fh[T]}\int_F(\vv\SCAL\normal_{TF})(\vw_F-\vw_T)\SCAL\vhv_T
    \right),
    $$
    we arrive at
    \begin{equation}\label{eq:th.consistency:basic}
      \begin{aligned}
        \int_\Omega\vw_h\trans\GRAD\vv\MVPROD\vv
        - t_h(\Ih\vv,\Ih\vv,\uw[h])
        &=
        \frac12\int_\Omega\vw_h\trans(\GRAD\vv-\Gh[2k]\uhv[h])\MVPROD\vv
        + \frac12\int_\Omega\vw_h\trans\Gh[2k]\uhv[h]\MVPROD(\vv-\vhv_h)
        \\
        &\quad + \frac12\int_\Omega(\vhv_h-\vv)\trans\GRADh\vw_h\MVPROD\vhv_h
        + \frac12\int_\Omega\vv\trans\GRADh\vw_h\MVPROD(\vhv_h-\vv)
        \\
        &\quad + \frac12\sum_{T\in\Th}\sum_{F\in\Fh[T]}\int_F((\vhv_T-\vv)\SCAL\normal_{TF})(\vw_F-\vw_T)\SCAL\vhv_T
        \\
        &\quad + \frac12\sum_{T\in\Th}\sum_{F\in\Fh[T]}\int_F(\vv\SCAL\normal_{TF})(\vw_F-\vw_T)\SCAL(\vhv_T-\vv).
      \end{aligned}
    \end{equation}
    Denote by $\term_1,\ldots,\term_6$ the addends in the right-hand side of the above expression.
    For the first term, using for all $T\in\Th$ the Euler equation \eqref{eq:GT.euler} with $\mtau=\vw_T\otimes\mproj[T]{0}\vv\in\Poly{k}(T)^{d\times d}$, we infer
    $$
    \term_1 = \frac12\int_\Omega\vw_h\trans(\GRAD\vv-\Gh[2k]\uhv[h])(\vv-\vproj{0}\vv).
    $$
    Hence, using the H\"{o}lder inequality followed by the approximation properties \eqref{eq:Gh.approx} of $\Gh[2k]$ and \eqref{eq:lproj.approx} of $\vproj{0}$ (with $m=0$, $p=4$, $s=1$), we obtain
    \begin{equation}\label{eq:th.consistency:T1}
      |\term_1|
      \lesssim h^{k+1}\norm[L^4(\Omega)^d]{\vw_h}\norm[H^{k+1}(\Omega)^d]{\vv}\norm[W^{1,4}(\Omega)^d]{\vv}
      \lesssim h^{k+1}\norm[1,h]{\uw[h]}\norm[H^{k+1}(\Omega)^d]{\vv}\norm[H^2(\Omega)^d]{\vv},
    \end{equation}
    where the conclusion follows using the discrete Sobolev embedding \eqref{eq:sobolev} with $p=4$ to bound the first factor and the continuous injection $H^2(\Omega)\to W^{1,4}(\Omega)$ valid in $d\in\{2,3\}$ on domains satisfying the cone condition to bound the third (cf. \cite[Theorem 4.12]{Adams.Fournier:03}).
    
    Using again the H\"{o}lder inequality, the boundedness \eqref{eq:Gh.cont} of $\Gh[2k]$ and \eqref{eq:Ih.cont} of $\Ih$ to infer $\norm[L^2(\Omega)^{d\times d}]{\Gh[2k]\uhv[h]}\lesssim\norm[H^1(\Omega)^d]{\vv}$, and the approximation properties \eqref{eq:lproj.approx} of $\vproj{k}$ (with $m=0$, $p=4$, and $s=k+1$), we infer
    \begin{equation}\label{eq:th.consistency:T2}
      |\term_2|
      \lesssim h^{k+1}\norm[L^4(\Omega)^d]{\vw_h}\norm[H^1(\Omega)^d]{\vv}\norm[W^{k+1,4}(\Omega)^d]{\vv}
      \lesssim h^{k+1}\norm[1,h]{\uw[h]}\norm[H^1(\Omega)^d]{\vv}\norm[H^{k+2}(\Omega)^d]{\vv},
    \end{equation}
    where the conclusion follows from the discrete Sobolev embedding \eqref{eq:sobolev} with $p=4$ together with the continuous injection $H^{k+2}(\Omega)\to W^{k+1,4}(\Omega)$ valid for all $k\ge 0$ and $d\in\{2,3\}$ on domains satisfying the cone condition (cf. \cite[Theorem 4.12]{Adams.Fournier:03}).
    
    Proceeding similarly, we have for the third and fourth terms
    \begin{equation}\label{eq:th.consistency:T3.T4}
      \begin{aligned}
        |\term_3| + |\term_4|
        &\lesssim h^{k+1}\norm[L^2(\Omega)^{d\times d}]{\GRADh\vw_h}\left(
        \norm[L^4(\Omega)^d]{\vhv_h} + \norm[L^4(\Omega)^d]{\vv}
        \right)\norm[W^{k+1,4}(\Omega)^d]{\vv}
        \\
        &\lesssim
        h^{k+1}\norm[1,h]{\uw[h]}\norm[H^1(\Omega)^d]{\vv}\norm[H^{k+2}(\Omega)^d]{\vv},
      \end{aligned}
    \end{equation}
    where, to pass to the second line, we have used the definition \eqref{eq:norm1h} of the $\norm[1,h]{{\cdot}}$-norm to bound the first factor, the discrete and continuous Sobolev embeddings to estimate the $L^4(\Omega)^d$-norms in the second factor, the boundedness \eqref{eq:Ih.cont} of $\Ih$ to further bound $\norm[1,h]{\uhv[h]}\lesssim\norm[H^1(\Omega)^d]{\vv}$, and the continuous injection $H^{k+2}(\Omega)\to W^{k+1,4}(\Omega)$ to conclude.

    Finally, for the fifth and sixth term, using H\"{o}lder inequalities and the trace approximation properties \eqref{eq:lproj.approx} of the $L^2$-orthogonal projector (with $m=0$, $p=4$, and $s=k+1$), we obtain
    \begin{equation}\label{eq:th.consistency:T5.T6}
      \begin{aligned}
        |\term_5| + |\term_6|
        &\lesssim h^{k+1}\norm[W^{k+1,4}(\Omega)^d]{\vv}\left(
        \sum_{T\in\Th}\seminorm[1,\partial T]{\uw[T]}^2
        \right)^{\nicefrac12}\hspace{-1.5ex}\times\left(
        \sum_{T\in\Th} h_T \left(\norm[L^4(\partial T)^d]{\vv}^4 + \norm[L^4(\partial T)^d]{\vhv_T}^4\right)
        \right)^{\frac14}
        \\
        &\lesssim h^{k+1}\norm[H^{k+2}(\Omega)^d]{\vv}\norm[1,h]{\uw[h]}\norm[H^2(\Omega)^d]{\vv},
      \end{aligned}
    \end{equation}
    where, to pass to the second line, we have used the continuous injection $H^{k+2}(\Omega)\to W^{k+1,4}(\Omega)$ for the first factor, the definition \eqref{eq:norm1h} of the $\norm[1,h]{{\cdot}}$-norm for the second factor, and the continuous \eqref{eq:trace.cont} and discrete \eqref{eq:trace} trace inequalities with $p=4$ followed by the continuous injection $H^2(\Omega)\to W^{1,4}(\Omega)$ for the third factor.
    Taking absolute values in \eqref{eq:th.consistency:basic}, and using \eqref{eq:th.consistency:T1}--\eqref{eq:th.consistency:T5.T6} to bound the right-hand side, \eqref{eq:th.consistency} follows.
  \item \emph{Sequential consistency.}
    We have, letting for the sake of brevity $\uvhPhi\eqbydef\Ih\vPhi$,
    $$
    t_h(\uv[h],\uv[h],\uvhPhi) =
    \frac12\int_\Omega\vhPhi_h\trans\Gh[2k]\uv[h]\MVPROD\vv_h
    - \frac12\int_\Omega\vv_h\trans\Gh[2k]\uvhPhi\MVPROD\vv_h
    \eqbydef\term_1 + \term_2.
    $$
    Since $\vv_h\to\vv$ and $\vhPhi_h\to\vPhi$ strongly in $L^4(\Omega)^d$, $\vhPhi_h\otimes\vv_h\to\vPhi\otimes\vv$ strongly in $L^2(\Omega)^{d\times d}$. Hence, recalling that $\Gh[2k]\uv[h]\rightharpoonup\GRAD\vv$ weakly in $L^2(\Omega)^{d\times d}$ owing to point 4) in Proposition \ref{prop:Gh.properties}, we infer that
    $$
    \term_1=\frac12\int_\Omega\Gh[2k]\uv[h]\SSCAL (\vhPhi_h\otimes\vv_h)\to\frac12\int_\Omega\vPhi\trans\GRAD\vv\MVPROD\vv.
    $$
    For the second term, observing that $\vv_h\otimes\vv_h\to\vv\otimes\vv$ and $\Gh[2k]\uvhPhi\to\GRAD\vPhi$ strongly in $L^2(\Omega)^{d\times d}$, we readily get
    $$
    \term_2=\frac12\int_\Omega\Gh[2k]\uvhPhi\SSCAL (\vv_h\otimes\vv_h)\to\frac12\int_\Omega\vv\trans\GRAD\vPhi\MVPROD\vv.
    $$
    The conclusion follows from the above results recalling the definition \eqref{eq:a.b.t} of $t$.
    \qedhere
  \end{asparaenum}
\end{proof}

\subsection{Velocity-pressure coupling bilinear form}\label{sec:bh.properties}

\begin{proof}[Proof of Proposition \ref{prop:bh.properties}]
  \begin{asparaenum}[1)]    
  \item \emph{Inf-sup stability.} We deploy similar arguments as in \cite[Lemma 4]{Boffi.Botti.ea:16} and \cite[Lemma 3]{Di-Pietro.Ern.ea:16}.
    Let $q_h\in\Ph$ and denote by $\mathrm{S}$ the supremum in \eqref{eq:bh.inf-sup}.
    Observing that $q_h\in P$, from the surjectivity of the continuous divergence operator from $\U$ to $P$ we infer the existence of $\vv_{q_h}\in\U$ such that $\DIV\vv_{q_h}=q_h$ and $\norm[H^1(\Omega)^d]{\vv_{q_h}}\lesssim\norm[L^2(\Omega)]{q_h}$.
    Then, we have
    $$
    \norm[L^2(\Omega)]{q_h}^2
    = -b_h(\Ih\vv_{q_h},q_h)
    \le\mathrm{S}\norm[1,h]{\Ih\vv_{q_h}}
    \lesssim\mathrm{S}\norm[H^1(\Omega)^d]{\vv_{q_h}}
    \lesssim\mathrm{S}\norm[L^2(\Omega)]{q_h},
    $$
    where we have used the commuting property \eqref{eq:Dh.commuting} for $\Dh$, the definition of the supremum, the boundeness \eqref{eq:Ih.cont} of $\Ih$, and $\norm[H^1(\Omega)^d]{\vv_{q_h}}\lesssim\norm[L^2(\Omega)]{q_h}$ to conclude.

  \item \emph{Consistency.}
    Integrating by parts element-by-element, and using the fact that the jumps of $q$ vanish at interfaces by the assumed regularity and that $\vv_F=\vec{0}$ on boundary faces to insert $\vv_F$ into the second term, we have
    \begin{equation}\label{eq:bh.consistency:1}
      \int_\Omega\GRAD q\SCAL\vv_h = -\sum_{T\in\Th}\left(
      \int_T q (\DIV\vv_T) + \sum_{F\in\Fh[T]}\int_F q (\vv_F - \vv_T)\SCAL\normal_{TF}
      \right).
    \end{equation}
    On the other hand, using \eqref{eq:DT'} on each element $T\in\Th$ to express the right-hand side of \eqref{eq:bh}, we have
    \begin{equation}\label{eq:bh.consistency:2}
      -b_h(\uv[h],\lproj{k}q)
      = \sum_{T\in\Th}\left(
      \int_T \lproj[T]{k}q (\DIV\vv_T) + \sum_{F\in\Fh[T]}\int_F \lproj[T]{k} q (\vv_F - \vv_T)\SCAL\normal_{TF}
      \right).
    \end{equation}
    Summing \eqref{eq:bh.consistency:1} and \eqref{eq:bh.consistency:2}, observing that the first terms in parentheses cancel out by the definition \eqref{eq:lproj} of $\lproj[T]{k}$ since $(\DIV\vv_T)\in\Poly{k-1}(T)\subset\Poly{k}(T)$, and using for the second terms the Cauchy--Schwarz inequality followed by the trace approximation properties \eqref{eq:lproj.approx} of $\lproj[T]{k}$, we infer that
    $$
    \left|
    \int_\Omega\GRAD q\SCAL\vv_h - b_h(\uv[h],\lproj{k}q)    
    \right|
    \le \left(\sum_{T\in\Th} h_T\norm[L^2(\partial T)]{\lproj[T]{k}q - q}^2\right)^{\nicefrac12}\hspace{-1.25ex}\times\left(
    \sum_{T\in\Th}\seminorm[1,\partial T]{\uv[T]}^2
    \right)^{\nicefrac12}
    \lesssim h^{k+1}\norm[H^{k+1}(\Omega)]{q}\norm[1,h]{\uv[h]}.
    $$
    Passing to the supremum in the above expression, \eqref{eq:bh.consistency} follows.
    
  \item \emph{Sequential consistency.}
    Recalling~\eqref{eq:DT.trGT}, $\Dh=\tr(\Gh)$ and the sequential consistency \eqref{eq:bh.seq.cons:1} is a straightforward consequence of point 2) in Proposition \ref{prop:Gh.properties} combined with a weak-strong convergence argument.
    Similarly, the sequential consistency \eqref{eq:bh.seq.cons:2} follows from the fact that $\Dh\uv[h]\rightharpoonup\DIV\vv$ weakly in $L^2(\Omega)$
    as a consequence of point 3) in Proposition \ref{prop:Gh.properties} and $\lproj{k}\varphi\to\varphi$ strongly in $L^2(\Omega)$.
    \qedhere
  \end{asparaenum}
\end{proof}

%------------------------------------------------------------------------------%

\begin{small}
  \bibliographystyle{plain}
  \bibliography{nsho}
\end{small}

\end{document}